\documentclass[a4paper,12pt]{article}
\usepackage[utf8]{inputenc}
\usepackage[english]{babel}
\usepackage{amsthm}
\usepackage{amssymb}
\usepackage{amsmath}
\usepackage{mathrsfs}
\usepackage{amssymb}
\usepackage{pdfsync}
\usepackage{comment}
\usepackage{mathtools}
\usepackage{graphicx, color}
\usepackage{tikz,float}
\usetikzlibrary{shapes,arrows,shadows}

\newtheorem{thm}{Theorem}[section]
\newtheorem{corollary}[thm]{Corollary}
\newtheorem{definition}[thm]{Definition}
\newtheorem{lemma}[thm]{Lemma}
\newtheorem{example}[thm]{Example}
\newtheorem{rmk}[thm]{Remark}
\title{An extension theorem from connected sets\\ and homogenization of non-local functionals }
\author{ Andrea Braides\footnote{Dipartimento di Matematica
 Universit\`a di Roma ``Tor Vergata'', via della ricerca scientifica 1, 00133 Rome, Italy {\tt e-mail:braides@mat.uniroma2.it}}\hspace{0.2cm}, Valeria Chiad\`o Piat  and  Lorenza D'Elia\footnote{Dipartimento di Scienze Matematiche, Politecnico di Torino
 Corso Duca degli Abruzzi 24,
10129 Torino, Italy {\tt e-mail:valeria.chiadopiat@polito.it, lorenza.delia@polito.it}}
}

\date{}                                           

\def\e{\varepsilon}
\def\xe{{\tfrac{x}{\varepsilon}}}
\def\R{{\mathbb R}}
\def\Z{{\mathbb Z}}
\def\N{{\mathbb N}}

\def\hom{{\rm hom}}
\def\loc{{\rm loc}}

\def\tto{{\buildrel{\tau}\over{\longrightarrow}}} 
\def \trait (#1) (#2) (#3){\vrule width #1pt height #2pt depth #3pt}
\def \qed{\hfill
        \trait (0.1) (6) (0)
        \trait (6) (0.1) (0)
        \kern-6pt
        \trait (6) (6) (-5.9)
        \trait (0.1) (6) (0)
\medskip}

\DeclareMathOperator{\dist}{dist}

\def\a{\alpha}
\def\omp{\Omega^\prime}
\def\meu{\overline{u}}
\def\ua {u^\alpha}
\def\psia{\psi^\alpha}
\def\ub{u^\beta}
\def\rd{\mathbb{R}^d}

\def\calR{{\cal R}}

\def\DR{{D_R}}

\begin{document}
\maketitle

\smallskip
\noindent
{\bf Abstract.} 
{We study the asymptotic behaviour of convolution-type functionals defined on general periodic domains by proving an extension theorem.

\smallskip
\noindent
{\bf Keywords:} homogenization, perforated domains, non-local functionals, extension operators}

\smallskip
\noindent
{\bf AMS Classifications.} 49J45, 49J55, 74Q05, 35B27, 35B40, 45E10

\section{Introduction}
In this paper we consider energies of convolution-type whose prototypes are functionals of the form
      \begin{equation}
      \label{energies}
         \frac{1}{\varepsilon^{d+p}}\int_{\Omega\times\Omega} a\left(\frac{y-x}{\varepsilon}   \right)|u(y)-u(x)|^pdx\,dy,
         \end{equation}
where $a$ is a non-negative convolution kernel,  $p\in (1,+\infty)$, $\varepsilon$ is a scaling parameter and $\Omega$ is a Lipschitz domain in $\rd$. 
The kernel $a:\rd\rightarrow [0,+\infty[$, 
 describing  the strength of the interaction at a given distance,  satisfies
\begin{equation}
    \label{c1}
    \int_{\R^d}a(\xi)(1+|\xi|^{p})\, d\xi<+\infty,
    \end{equation}
and
     \begin{equation}
     \label{c2}
     a(\xi)\geq c>0, \qquad\text{if}\quad |\xi|\leq r_0,
     \end{equation}
for some $r_0>0$ and $c>0$.

Functionals of this form have been used as an approximation of the $L^p$-norm of the gradient as $\e\to 0$ and as such give an alternative way of defining Sobolev spaces (see {\it e.g.}~\cite{AB,BBM}).
In the case $p=2$ perturbations of such energies \eqref{energies} arise from models in population dynamics where the macroscopic properties are reduced to studying the evolution of the first-correlation function describing the population density $u$ in the system \cite{FKK}, and recently they have also been used in problems in Data Science \cite{GS}. Furthermore discrete versions of such energies have been extensively studied in a general setting (see {\it e.g.}~\cite{AC,BLL} and related works).

A rather complete analysis of perturbations of functionals \eqref{energies}, more precisely, of functionals that are dominated from below and above by functionals of type \eqref{energies}, is presented in \cite{AABPT}. 
In this paper we consider another type of perturbation of \eqref{energies} in the framework of the so-called {\em perforated domains}, that cannot be reduced to the analysis in \cite{AABPT} since it is `degenerate' on the complement of a periodic connected set.

In our analysis we consider a typical situation arising in the study of inhomogeneous media with a periodic microstructure, when one sets the model in a domain obtained by removing inclusions representing sites with which the system does not interact. Usually, such a periodically perforated domain is obtained by intersecting $\Omega$ with a periodic open subset $E_\delta = \delta E$ of $\rd$, where $E$ is a periodic set with Lipschitz boundary and $\delta$ is the (small) period of the microstructure. In the setting of energies \eqref{energies} the relevant scale of the period $\delta$ is of order $\varepsilon$. Indeed, in the other cases we have a multi-scale problem that can be decomposed into two separate limit analyses that fall within known results corresponding to letting first $\delta\to0$ and then $\e\to0$, or the converse (see \cite{BP_per}). Hence, we will consider energies whose prototypes are of the form
      \begin{equation}
      \label{funperiod}
      F_{\varepsilon}(u) = \frac{1}{\varepsilon^{d+p}}\int_{(\Omega\cap \varepsilon E)\times (\Omega\cap \varepsilon E)} a\left(\frac{y-x}{\varepsilon}\right)|u(x)-u(y)|^pdy\,dx,
      \end{equation}
where $\Omega$ is a fixed domain in $\R^d$.
      
 In order to study the asymptotic analysis of such energies, it is necessary to prove that sequences with equi-bounded energy (and equi-bounded $L^p$-norm) are precompact. For the analog energy on Sobolev spaces
         \begin{equation*}
      F^{\rm Sob}_{\varepsilon}(u) = \int_{\Omega\cap \varepsilon E}|\nabla u|^pdy\,dx.
      \end{equation*}  this has been done in \cite{ACDP} through the construction of suitable extension operators $T_\e: 
      L^p(\Omega\cap\e E)\to L^p(\Omega)$ which, for each  $\Omega'$ compactly contained in $\Omega$, provide an embedding of $W^{1,p}(\Omega')$ in $W^{1,p}(\Omega\cap\e E)$ uniformly for $\e$ small enough (below a threshold explicitly depending on the distance between $\Omega'$ and $\partial\Omega$). The compact embedding of $W^{1,p}(\Omega')$ in $L^p(\Omega')$ then provides the desired compactness property.
       In our case, since the energies are non-local, a more complex statement is necessary. After noting that by condition \eqref{c2} it is sufficient to prove compactness when $a$ is the characteristic function of a ball centered in $0$ and given radius $r_0$, we prove the existence of extension operators $T_\e: L^p(\Omega\cap\e E)\to L^p(\Omega)$  with the property that $R$ and $C$ exists such that for each $\Omega'$ compactly contained in $\Omega$,
         \begin{eqnarray}\nonumber
&&\int_{\Omega'\times \Omega'} \chi_{B_R}\left(\frac{y-x}{\varepsilon}\right)|T_\e u(x)-T_\e u(y)|^pdy\,dx\\
&&\le C
\int_{(\Omega\cap \varepsilon E)\times (\Omega\cap \varepsilon E)} \chi_{B_{r_0}}\left(\frac{y-x}{\varepsilon}\right)|u(x)-u(y)|^pdy\,dx,\label{imminent}
\end{eqnarray}
for $\e$ small enough, with $C$ and $R$ independent of $\e$ (here $B_\rho$ denotes the ball of centre $0$ and radius $\rho$ and $\chi_A$ is the characteristic function of the set $A$). 
The precise statement  of this result is given in Theorem \ref{ext-thm}. 
It provides a uniform bound for energies of the type \eqref{energies} on $\Omega'$ in terms of energies \eqref{funperiod}, which in turn allows to apply the compactness results in \cite{AABPT} (see Section \ref{sect-comp}). Moreover, the asymptotic analysis of functionals \eqref{energies} ensure that limits of functions with equibounded energies are in $W^{1,p}(\Omega')$ with a uniform bound and hence they belong to $W^{1,p}(\Omega)$. 

The case $p=2$ in \eqref{funperiod} and with compact perforations; {\it i.e.}, with $E$ of the form $E=\rd\setminus (K_0+\mathbb{Z}^d)$, where $K_0$ is a compact subset of $\rd$ with Lipschitz boundary such that $(K_0+i)\cap (K_0+j)=\emptyset$ if $i,j\in\mathbb{Z}^d$ and $i\neq j$, has been studied in \cite{BP_per}, together with some variants that allow to consider random perforations \cite{BP_ran}. The main feature of our paper is the proof of the extension theorem under the only assumption that the periodic set $E$ is connected and with Lipschitz boundary, and holds for any $p>1$. The construction of $T_\varepsilon$ is inspired by the arguments of \cite{ACDP}, consisting in proving a local extension result on cubes and then using a periodic partition of  the unity. The non-locality of the energies adds further technical difficulties to the possible non-connectedness or non-regularity of the restriction of $E$ to cubes, already present in the case of Sobolev functions, and forces the introduction of the radius of interaction $R$ in inequality \eqref{imminent}. 

As an application, we study the asymptotic behaviour of energies of the form 
   \begin{equation*}
   H_\e(u) = {1\over \e^d}\int_{(\Omega\cap \e E)^2} h\left({x\over\e},{y\over \e},\frac{u(y)-u(x)}{\e}\right)dx\,dy,
   \end{equation*}
with $u\in L^p(\Omega;\R^m)$, upon some structure hypotheses on $h$  as those considered in \cite{AABPT}, that allow $H_\e$ to be compared with $F_\e$. In Section \ref{sect-hom} we obtain a homogenization theorem for $H_\e$ as $\e\to 0$ proving that the $\Gamma$-limit of $H_\e$ is defined on $W^{1,p}(\Omega;\R^m)$ and has a standard local form 
$$
\int_\Omega h_{\rm hom}(Du)dx,
$$
with $h_{\rm hom}$ characterized by non-local homogenization formulas and of $p$-growth by \eqref{c1} and \eqref{c2}. The proof is obtained by a perturbation argument that allows to use homogenization theorems proved in \cite{AABPT} for the corresponding energies defined on `solid' domains, applied to functionals of the form $H_\e+\delta F_\e$. The Extension Theorem  provides uniform estimates that allow to invert the passage to the limit as $\e\to 0$ and $\delta\to0$. We note that a discrete analog of this result can be found in \cite{BCP}, where the discrete setting allows easier extension results from the discrete version of a perforated domain.



\smallskip
Before stating and proving the main result we gather some of the notation used in the following.

\subsubsection*{Notation}\label{sect-setup}
\begin{itemize}
   \item $Q= (0, 1)^d$ denotes the unit cube in $\rd$.
   \item $\chi_A$ denotes the characteristic function of the set $A$.
   \item $\lfloor t\rfloor$ denotes the integer part of $t\in\R$.
   \item  $\mathbf{M}^{m\times d}$ is the space of $m\times d$ real matrices.
   \item if $\Xi\in\mathbf{M}^{m\times d}$ and $x\in\R^d$ then $\Xi x\in\R^m$ is defined by the usual row-by-column product. 
   \item For any open set $\Omega\subset\R^d$ and for any $\lambda>0$, $\lambda\Omega$ denotes the $\lambda$-homothetic set
             \begin{equation*}
             \lambda\Omega := \{\lambda x\hspace{0.03cm}:\hspace{0.03cm}x\in\Omega\},
             \end{equation*}
   and $\Omega(\lambda)$ is the retracted set 
        \begin{equation}\label{retract}
        \Omega(\lambda) :=\{ x\in\Omega\hspace{0.03cm}:\hspace{0.03cm}  \dist (x, \partial\Omega) >\lambda\}.
        \end{equation} 
   \item For $R>0$, $D_R$ denotes the set of points in $\R^d\times \R^d$ whose distance is less than $R$; \textit{i.e.}, 
       \begin{equation*}
       D_R:= \{(x,y)\in\R^d\times\R^d\hspace{0.03cm}:\hspace{0.03cm}|x-y|\leq R \}.
       \end{equation*}
   \item Given an open set with finite Lebesgue measure $|A|<\infty$, the mean value of $u$ over $A$ is given by 
   \begin{equation}\label{mean}
   u_A=\frac1{|A|}\int_Au(x)\, dx.
   \end{equation}

\item We say that a set $E\subset\mathbb{R}^d$ is periodic (more precisely, $Q$-periodic) if $E+e_i=E$ for every $i=1,2,\cdots, d$ where $(e_i)_{i=1}^{d}$ is the canonical basis of $\mathbb{R}^d$.
   \end{itemize}
    
%


%

\section{The extension theorem}\label{sect-lemmi}
In this section, we prove the existence of an extension operator
for non-local functionals defined on general connected domains. The main result of the paper is Theorem \ref{ext-thm}, from which we deduce a compactness result in Section \ref{sect-comp}.
Before stating it, we recall the definition of a set with Lipschitz boundary.

\begin{definition} An open set $E\subset\mathbb{R}^n$ has {\em Lipschitz boundary at $x\in\partial E$} if $\partial E$ is locally the graph of a Lipschitz function, in the sense that there exist a coordinate system $(y_1,\dots, y_d)$, a Lipschitz function $\Phi$ of $d-1$ variables, and an open rectangle $U_x$ in the $y$-coordinates, centred at $x$, such that $E\cap U_x=\{y\hspace{0.02cm} : \hspace{0.02cm} y_n<\Phi(y_1,\dots,y_{d-1})   \}$ and that $\partial E$ splits $U_x$ into two connected sets, $E\cap U_x$ and $U_x\setminus\overline{E}$. If this property holds for every $x\in\partial E$ with the same Lipschitz constant, we say that $E$ has {\em Lipschitz boundary}.
\end{definition}

   \begin{thm}\label{ext-thm}
   Let $E$ be a periodic open subset of $\R^d$ with Lipschitz boundary and let $\Omega$ be a bounded open subset of $\R^d$. Then, there exist $R=R(E)>0$  and  $k_0>0$ such that for all $\e>0$ there exists a linear and continuous extension operator
   $T_\e:L^p(\Omega\cap\e E)\to L^p (\Omega)$ such that for all $r>0$ and
   for all $u\in L^p(\Omega\cap\e E)$,
   \begin{equation}\label{1e}
   T_\e(u)=u \quad{\rm  a.e.~\hspace{0.2cm} in}\hspace{0.3cm}\Omega\cap\e E,
   \end{equation}
   \begin{equation}\label{2e}
   \int_{\Omega(\e k_0)} |T_\e (u)|^p\, dx\le c_1\int_{\Omega\cap \e E} | u|^p\, dx,
   \end{equation}
   \begin{equation}\label{3e}
   \int_{(\Omega(\e k_0))^2\cap D_{\e R}}\left| T_\e (u)(x)-T_\e (u)(y)     \right|^p\, dx\, dy\le c_2(r)\int_{(\Omega\cap\e E)^2\cap D_{\e r}}\left|  u(x)- u(y)     \right|^p\, dx\, dy,
   \end{equation}
where we use notation \eqref{retract}.
The positive constants $c_1$ and $c_2$ depend on $E$ and $d$ and, in addition,  $c_2$ depends also on $r$, but both are independent of $\e$. 
   \end{thm}
The proof, which will be given in the next subsection, is quite technical and it is split into several lemmas.

\subsection{Technical lemmas and proof of the main result}
In order to give an idea of the construction of the extension operator, we assume that $E\cap 2Q$ is connected and has Lipschitz boundary. Under these assumptions, there exists a linear and continuous operator $\Phi: L^p(E\cap 2Q)\to L^p(2Q)$ satisfying, in particular, an estimate analogous to \eqref{3e} (see Lemma \ref{lemma2.6}). Then, we consider the family $\Phi^\alpha$ of the extension operator obtained by traslating $\Phi$ by an integer vector $\alpha\in\mathbb{Z}^d$. Finally, thanks to a periodic partition of unity, the construction of a global extension operator  is achieved glueing together $\Phi^\alpha$ (see Lemma \ref{lemma2.7}). Now, the assumptions that $E\cap 2Q$ is connected and has Lipschitz boundary in general are not satisfied (unless the complement of $E$ is a disjoint union of compact sets, which is the case studied in \cite{BP_per}), so that  the first step consists to overcome the lack of connectedness of $E\cap 2Q$ and the regularity of its boundary. To this end, we state a slightly modified version of  \cite[Lemma 2.3]{ACDP}, which is a key tool for the construction of the extension operator.  The proof remains analogous to that of \cite[Lemma 2.3]{ACDP} and is not repeated here. 
   \begin{lemma}
   \label{LemmaACMP}
   Let $E$ be a connected open subset of $\rd$ with Lipschitz boundary. Then, there exists $k\in\mathbb{N}$, $k\geq 4$, such that $3Q\cap E$ is contained in a single connected component $C$ of $kQ\cap E$. Moreover, $C$ has Lipschitz boundary at each point of $\partial C\cap 3\overline{Q}$.
   \end{lemma}
We denote henceforth by $\tilde{C}$ the positive constant given by  $\tilde{C}:= 2\sqrt{d}k$, where $k$ is defined as  Lemma \ref{LemmaACMP}.

The next lemma is an easy consequence of the H\"{o}lder inequality. 
\begin{lemma}\label{LSE}
Let $A$ be an open subset of $\R^d$. Assume that $A$ has finite and positive Lebesgue measure $|A|<\infty$. Then, for every $u\in L^p(A)$, with $1<p<\infty$,
    \begin{equation}
                \label{SE}
                \int_{A} |u_A-u(x)|^pdx\leq \frac{1}{|A|}\int_{A\times A} |u(x)-u(y)|^pdxdy.
                \end{equation}

\end{lemma}

\proof Denote by $p'$ the conjugate exponent of $p$. Thanks to H{\"o}lder's inequality, we deduce
     \begin{align*}
     \int_{A} |u_A-u(x)|^pdx&= \frac{1}{|A|^p}\int_{A}\left|\int_{A} (u(y)-u(x))dy\right|^pdx\\
     &\leq \frac{|A|^{p/p'}}{|A|^p}\int_{A}\int_{A} |u(y)-u(x)|^pdydx\\
     &= \frac{1}{|A|} \int_{A\times A}|u(y)-u(x)|^pdxdy,
     \end{align*}
which concludes the proof.
\qed

The next lemma shows the existence of an extension operator $\Phi$ on general sets of $\rd$. It is an adaptation of \cite[Lemma 2.6]{ACDP}.   

\begin{lemma}\label{lemma2.6}
Let $B$, $\omega$, $\omega'$ be bounded open subsets of $\R^d$. Assume that $\partial B$ is Lipschitz-continuous at each point of $\partial B\cap \overline\omega$ and $\omega'\subset\subset\omega$. Then, there exist a positive real number $R>0$ and a linear and continuous extension operator $\Phi:L^p(B)\to L^p(\omega')$ such that,
for all $u\in L^p(B)$,
\begin{equation}\label{1}
\Phi (u)=u \qquad\mbox{a.e.~in }\hspace{0.1cm}B\cap\omega',
\end{equation}
\begin{equation}\label{2}
\int_{\omega'} |\Phi (u)|^p\, dx\le c_1\int_{B\cap\omega}|u|^p\, dx,
\end{equation}
\begin{equation}\label{3}
\int_{(\omega'\times\omega')\cap D_R}\left| \Phi (u)(x)-\Phi (u)(y)     \right|^p\, dx\, dy\le c_2\int_{(B\cap\omega)^2}\left|  u(x)- u(y)     \right|^p\, dx\, dy,
\end{equation} 
where $c_1$ and $c_2$ are positive constant depending only on $B, \omega', \omega$ and $p$.
\end{lemma}

\proof
Since $\partial B$ has Lipschitz boundary at each point of $\partial B\cap\overline{\omega}$, there exist a neighbourhood $U$ of $\partial B\cap\overline{\omega}$ and a bi-lipschitz map ${\calR}:U\cap B\to U\setminus B$ such that, for any $x_1,x_2\in U\cap B$,
$$
\frac12|\calR(x_1)-\calR(x_2)|\le |x_1-x_2|\le 2
|\calR(x_1)-\calR(x_2)|.
$$
For fixed $t>0$ chosen below, we consider the set
\begin{equation}\label{At}
A_t:=\{ x\in \omega\setminus B:{\rm dist}(x,\partial B)<t      \}.
\end{equation}
We may fix $t>0$ small enough  such that 
      \begin{equation}
      \label{propA}
      A_t\cap\omega'\subset U\setminus B \qquad \text{and}\qquad \calR^{-1}( A_t\cap\omega') \subset B\cap\omega.
      \end{equation} 
Let $\varphi$ be a $C^\infty$ function such that $0\le \varphi\le 1$, $\varphi\equiv 1$ in $\overline B$ and $\varphi\equiv 0$ in 
$\{x\in \R^d\setminus B:{\rm dist}(x,\partial B)\ge t\}$.
We define the operator $\Phi\hspace{0.03cm}:\hspace{0.03cm} L^p(B)\to L^p(\omega')$ as follows
\begin{equation}
\label{defop}
\Phi (u)(x):=
\begin{dcases}
  u(x), & x\in B\cap\omega', \\
  \varphi(x) u(\calR^{-1}(x)) + (1-\varphi(x)) u_{B\cap\omega}, & x \in A_t\cap\omega',\\
  u_{B\cap\omega}, & x\in \omega'\setminus A_t,
\end{dcases}
\end{equation}
where $u_{B\cap\omega}$ denotes the mean value of the function $u$ over $B\cap\omega$ (see \eqref{mean}). It follows that $\Phi(u)\in L^p(\omega')$ and $\Phi(u)=u$ a.e.~in $B\cap\omega'$;  \textit{i.e.}, condition \eqref{1} is satisfied.

We now show condition \eqref{2}. To this end, note that $\omega'$ can be written as
     \begin{equation*}
     \label{propR}
     \omega'= (B\cap\omega')\cup (A_t\cap\omega')\cup(\omega'\setminus A_t).
     \end{equation*}
This, combined with the Jensen inequality and the definition \eqref{defop} of $\Phi$, yields
    \begin{align}
    \int_{\omega'} |\Phi (u)(x)|^pdx&=\int_{B\cap\omega'} |\Phi (u)(x)|^pdx + \int_{A_t\cap\omega'} |\Phi (u)(x)|^pdx + \int_{\omega'\setminus A_t} |\Phi (u)(x)|^pdx\notag\\
    &=\int_{B\cap\omega'} |u(x)|^pdx + \int_{A_t\cap\omega'} |\varphi(x) u(\calR^{-1}(x)) + (1-\varphi(x)) u_{B\cap\omega}|^pd\xi  \notag \\
    &\quad+ |\omega'\setminus A_t||u_{B\cap\omega}|^p\notag\\
    &\leq \int_{B\cap\omega'} |u(x)|^pdx + 2^{p-1}\int_{A_t\cap\omega'} | u(\calR^{-1}(x))|^pdx\notag\\
    &\quad + |u_{B\cap\omega}|^p (2^{p-1}|\omega'\cap A_t|+|\omega'\setminus A_t|).\label{es1}
    \end{align}
Since $\calR$ is a bi-Lipschitz map, the Jacobian $\left|\tfrac{\partial \calR}{\partial x}(x)\right|$ is a bounded function;  \textit{i.e.}, there exists a positive constant $c_\calR$ such that
     \begin{equation}
     \label{Jacobian}
     \left|\frac{\partial \calR}{\partial x}(x) \right|\leq c_\calR,
     \end{equation}
so that, thanks to the change of variables $x'=\calR^{-1}(x)$ and properties \eqref{propA}, we have
        \begin{align*}
        \int_{A_t\cap\omega'} |u(\calR^{-1}(x))|^pdx \leq c_\calR \int_{B\cap\omega} |u(x')|^pdx'.
        \end{align*}
This, along with \eqref{es1}, implies that 
     \begin{align*}
     \int_{\omega'} |\Phi (u)(x)|^pdx &\leq (c_{\calR}2^{p-1}+1)\int_{B\cap\omega'} |u(x)|^pdx + |u_{B\cap\omega}|^p (2^{p-1}|\omega'\cap A_t|+|\omega'\setminus A_t|)\\
    &\leq c_1 \int_{B\cap\omega} |u(x)|^pdx, 
     \end{align*}   
where $c_1$ denotes a positive constant depending only on $p, \omega', B$ and $\calR$. Hence, condition \eqref{2} is proven. 

To conclude the proof, it remains to check condition \eqref{3}. Fix  $R<t$. For   $(x, y)\in(\omega'\times\omega')\cap D_R$, it is enough to estimate the integral in the left-hand side of \eqref{3} by examining separately the sets
    \begin{align*}
    S_1 &= ((B\cap \omega')\times(B\cap \omega'))\cap \DR,\\
    S_2 &= ((B\cap \omega') \times (A_t\cap\omega') )\cap \DR,\\
    S'_2 &= (  (A_t\cap\omega')\times (B\cap \omega')\cap\DR,\\
    S_3 &=  ((A_t\cap\omega') \times (A_t\cap\omega'))\cap \DR,\\
    S_4 &=  ((A_t\cap\omega')\times (\omega'\setminus A_t))\cap \DR,\\
    S'_4 &= ( (\omega'\setminus A_t)\times(A_t\cap\omega'))\cap \DR,\\
    S_ 5 &= ((\omega'\setminus A_t)\times(\omega'\setminus A_t)   )\cap \DR.
    \end{align*}
Note that the other cases do not occur since the distance between the points is grater than $R$. Indeed, take, for example,  $(x,y)\in (B\cap\omega')\times(A_t\setminus \omega')$. Due to definition of $A_t$ and since $R<t$, the distance $|x-y|$ is greater than $R$.

\noindent Now, we evaluate the left-hand side of \eqref{3} on the set $S_i$ defined above. In view of the definition \eqref{defop} of $\Phi$, we have
   \begin{align}
    \int_{S_1}\left|\Phi (u)(x)-\Phi (u)(y) \right|^p\, dx dy & = \int_{S_1}\left|u(x)-u(y) \right|^p\, dx dy  \notag\\
    &\leq\int_{(B\cap\omega)^2}\left|u(x)-u(x) \right|^p\, dx dy. \notag
    \end {align}
Here, we used the fact that $S_1\subset (B\cap \omega')^2\subset (B\cap\omega)^2$.\\
Due to definition  \eqref{defop} of $\Phi$, an application of Jensen's inequality yields 
     \begin{align}
     \int_{S_2}|\Phi(u)(x)-\Phi(u)(y)|^pdx dy &=\int_{S_2} |u(x)-u(\calR^{-1}(y)) + (1-\varphi(y))(u(\calR^{-1}(y)) - u_{B\cap\omega})|^pdxdy\notag\\
      & \le 2^{p-1} \int_{S_2} |u(x)-u(\calR^{-1}(y))|^p \, dx dy   \nonumber \\
       &\quad+  2^{p-1} \int_{S_2} |1-\varphi(y)|^p \left|
       u(\calR^{-1}(y))-u_{B\cap\omega}   \right|^p\, dx dy   \label{i1}
        \end {align}
Using the  change of variables $y'=\calR^{-1}(y)$ and properties \eqref{propA} and \eqref{Jacobian}, the first integral in the left-hand side of \eqref{i1} can be estimates as 
\begin{align}
   \int_{S_2} |u(x)-u(\calR^{-1}(y))|^p \, dx dy & \leq \int_{B\cap\omega'}\left( \int_{A_t\cap\omega'}|u(x)-u(\calR^{-1}(y))|^pdy \right) dx\notag\\
   &\le  c_\calR  \int_{(B\cap\omega')^2} \left|u(x)-u(y')\right|^p dx dy' 
    \nonumber \\  &  \le c_\calR  \int_{(B\cap\omega)^2}
   \left|u(x)-u(y)\right|^p
  \,  dx dy.\label{i2}
  \end {align}
By applying Lemma \ref{LSE} and taking into account condition \eqref{Jacobian}, the second integral in the right-hand side of \eqref{i1} can be estimated as
 \begin{align}
 \int_{S_2} |1-\varphi(y)|^p \left|
   u(\calR^{-1}(y))-u_{B\cap\omega}   \right|^p\, dx dy
    & \le |B\cap\omega'|\int_{A_t\cap\omega'}  \left|
       u(\calR^{-1}(y))-u_{B\cap\omega}   \right|^p\, dy\notag\\ 
       &\leq c_\calR |B\cap\omega'| \int_{B\cap\omega} \left|
    u(y')-u_{B\cap\omega}
    \right|^p\,dy' \notag\\
   &\le c_\calR\frac{|B\cap\omega'|}{|B\cap\omega|}\int_{(B\cap\omega)^2}  \left|
      u (x)-u(y)   \right|^p\, dx dy.\notag
 \end{align}
Combined with \eqref{i1} and \eqref{i2}, this  implies 
         \begin{equation*}
         \int_{S_2} |\Phi u(x) - \Phi u(y)|^pdx dy\leq c\int_{(B\cap\omega)^2} | u(x) -  u(y)|^pdx dy,
         \end{equation*}
where $c$ is a positive constant depending on $p,B, \omega, \omega'$ and $\calR$.
Similarly, we have that 
       \begin{equation*}
       \int_{S'_2} |\Phi u(x) - \Phi u(y)|^pdx dy\leq c\int_{(B\cap\omega)^2} | u(x) -  u(y)|^pdx dy.
       \end{equation*}
Now, consider $(x, y)\in S_3$. From the definition \eqref{defop} of $\Phi$, we have
\begin{align}
\label{PhiS3}
\Phi u(x)-\Phi u(y) = F_1(x,y) + F_2(x,y),   
 \end{align}
where $F_1(x,y)$ and $F_2(x,y)$ are given by 
     \begin{align*}
     F_1(x, x)&:=( u(\calR^{-1}(x))-u_{B\cap\omega})( \varphi(x)- \varphi(y)),\\
     F_2(x, y)&:=\varphi(y) \left(  u(\calR^{-1}(x)) - u(\calR^{-1}(y))    \right).
     \end{align*}
Thanks to Lemma \ref{LSE} and due to properties \eqref{propA} and the estimate $|\varphi(x) - \varphi(y)|\leq 2$, we deduce that 
  \begin{align}
\int_{S_3} |F_1(x, y)|^p \, dx dy &\le 2^p \int_{(A_t\cap\omega')^2}\left|
u(\calR^{-1}(x))-u_{B\cap\omega}
\right|^p\, dxdy   \nonumber \\
&=2^p|A_t\cap\omega'|\int_{(A_t\cap\omega')}\left|
 u(\calR^{-1}(x))-u_{B\cap\omega}
\right|^p\, dx   \nonumber \\
&\le 2^p|A_t\cap\omega'| c_\calR  \int_{B\cap\omega}
 \left|
u(x')-u_{B\cap\omega} 
\right|^p \, dx'  \nonumber\\
& \le  2^p c_\calR\frac{|A_t\cap\omega'|}{|B\cap\omega|}  \int_{(B\cap\omega)^2} |u(x')-u(y)|^p\, dx' dy.  \label{f2}
 \end{align}
On the other hand, using the changes of variables $x'=\calR^{-1}(x)$ and $y'=\calR^{-1}(y)$, we get 
    \begin{align}
 \int_{S_3} |F_2(x, y)|^p \, dx dy  &\le  \int_{(A_t\cap\omega')^2}
 \left| u(\calR^{-1}(x)) - u(\calR^{-1}(y))    \right|^p \, dx dy  
 \nonumber\\
 & \le c_\calR^2\int_{(B\cap\omega)^2}\left| u(x')-u(y')\right|^p\, dx' dy'.    \label{f1}
 \end{align}
In view of \eqref{PhiS3}, an application of Jensen's inequality combined with \eqref{f2}
and \eqref{f1} leads to
 \begin{align}
\int_{S_3}  \left|
   \Phi (u)(x)-\Phi (u)(y)   \right|^p\, dx dy &\le 2^{p-1} \left(\int_{S_3} |F_1(x, y)|^p \, dx dy +   \int_{S_3} |F_2(x, y)|^p \, dx dy\right)\notag\\
   &\leq c\int_{(B\cap\omega )^2}\left| u(x)-u(y)\right|^p\, dx dy,\label{f3}
 \end{align}
where $c$ denotes a positive constant depending only on $p, B, \omega, \omega'$ and $\calR$.\\
Take now $(x,y)\in S_4$. Applying Lemma \ref{LSE} and using the change of variables $x'=\calR^{-1}(x)$, from the definition \eqref{defop} of $\Phi$, we deduce that
 \begin{align}
\int_{S_4}
\left| \Phi (u)(x)-\Phi (u)(y)
\right|^p\, dx dy       
& =|\omega'\setminus A_t|  \int_{A_t\cap\omega'}  \left| u(\calR^{-1}(x)) -u_{B\cap\omega}\right|^p \, dx  \nonumber\\
& \le c_\calR |\omega'\setminus A_t| \int_{B\cap\omega} \left| u(x')-u_{B\cap\omega} \right|^p\, dx' \nonumber\\
&\le c_\calR  \frac{|\omega'\setminus A_t|}{|B\cap\omega|}\int_{(B\cap\omega)^2}\left| u(x)-u(y)\right|^p\, dx dy\nonumber
 \end{align}
Similarly, we also get 
    \begin{equation*}
    \int_{S'_4}
    \left| \Phi (u)(x)-\Phi (u)(y)
    \right|^p\, dx dy  \leq  c_{\calR}\frac{|\omega'\cap A_t|}{|B\cap\omega|}\int_{(B\cap\omega)^2}\left| u(x)-u(y)\right|^p\, dx dy.
    \end{equation*}
Now, take $(x,y)\in S_ 5$. Hence, we have that $\Phi(x)-\Phi(y) = 0$ for a.e.~$x,y\in \omega'\setminus A_t$. Finally, gathering all the previous estimates, we conclude that 
    \begin{align*}
    \int_{(\omega'\times\omega')\cap D_R} |\Phi(u)(x) - \Phi(u)(y)|^pdxdy &= \sum_{i=1}^{5}\int_{S_i} |\Phi(u)(x) - \Phi(u)(y)|^pdxdy \\
    &\quad +\int_{S'_2\cup S'_4} |\Phi(u)(x) - \Phi(u)(y)|^pdxdy\\
    &\leq c_2 \int_{(B\cap\omega)^2}|u(x)-u(y)|^pdxdy,
    \end{align*}
where $c_2$ is a costant depending on $p, \omega', \omega$ and $B$. This shows \eqref{3} and concludes the proof.
\qed

The reflection argument that we used to construct the operator $\Phi$ cannot be used to prove the existence of a map $\Phi: L^p(B) \to L^p(\omega)$ since estimate \eqref{3} may not hold with $\omega'=\omega$, as showed in the following example.
 \begin{example} 
 Let $B$ be the ball in $\R^2$ centered at $0$  and of radius $1$ and let $\omega$ be the set of $\R^2$ defined by 
          \begin{equation*}
          \omega:= \{(x,y)\in\R^2\hspace{0.03cm}:\hspace{0.03cm} x\in(-1, 2), -x+1\leq y\leq -x+2 \}.
          \end{equation*}
   We define $u\in L^p(B)$ as
       \begin{equation*}
       u(x):=
       \begin{cases}
       1,&\quad x\in B\setminus \omega,\\
       0,&\quad x\in B\cap \omega.
       \end{cases}
       \end{equation*}
    If $\Phi(u)$ is the extension of $u$ out of $B$ by reflection, then we have  
      \begin{align*}
      \int_{\omega^2\cap D_R}|\Phi u(x) - \Phi u (y)|^pdx dy>0,
      \end{align*}
   since $u$ is not identically constant in the neighbourhood of the points $(1,0)$ and $(0,1)$,
   while
         \begin{align*}
              \int_{(B\cap\omega)^2\cap D_R}| u(x) -  u (y)|^pdx dy=0,
              \end{align*}
  so that the condition \eqref{3} is not satisfied.
  \end{example}

      \begin{lemma}\label{lemma2.7}
      Let $E$ be a periodic, connected open subset of $\R^d$ with Lipschitz boundary. Let $\Omega, \omp$ be open subsets of $\R^d$ such that $\omp\subset\subset\Omega$ and $\dist(\omp, \partial\Omega)>\tilde{C}$. Then there exist $R=R(E)>0$ and a linear and continuous operator 
                  $$
              L\hspace{0.02cm}:\hspace{0.02cm} L^p(\Omega\cap E)\rightarrow L^p(\omp)
              $$
      such that for all $r>0$ and for all $u\in L^p(\Omega\cap E),$
           \begin{equation}
           \label{Iestimate}
           Lu=u,\qquad\mbox{a.e.~in} \hspace{0.3cm}\omp\cap E,
           \end{equation}
           \begin{equation}
           \label{IIestimate}
           \int_{\omp} |Lu|^pdx\leq c_1\int_{\Omega\cap E} |u|^pdx,           
           \end{equation}
           \begin{equation}
           \label{stimagrad}
           \int_{(\omp\times\omp)\cap D_R} |Lu(x)-Lu(y)|^pdx dy\leq c_2(r) \int_{(\Omega\cap E)^2\cap D_{r}}|u(x)-u(y)|^pdx dy,           
           \end{equation}
       where $c_1$ and $c_2$ are positive constants depending on $E$ and $d$ and, in addition, $c_2$ depends  also on $r$. The constant $R$ depends only on the set $E$.
      \end{lemma}
      \begin{proof}
      In view of Lemma \ref{LemmaACMP}, there exists $k\in\N$, $k\geq 4$, such that $3Q\cap E$ is contained in a single connected component $C$ of $kQ\cap E$. Since $C$ has Lipschitz boundary at each point of $C\cap 3\overline{Q}$, we can apply Lemma \ref{lemma2.6} with $B=C$, $\omega'=2Q$ and $\omega=3Q$. Hence, there exist $R>0$ and a linear and continuous operator $\Phi\hspace{0.02cm}:\hspace{0.02cm} L^p(C)\rightarrow L^p(2Q)$ defined by \eqref{defop} such that, for any $u\in L^p(C)$,
           \begin{equation}
           \label{stima1cubetto}
           \Phi(u) = u \qquad\mbox{a.e.~in}\hspace{0.3cm} C\cap 2Q,
           \end{equation}
           \begin{equation}
           \int_{2Q} |\Phi(u)|^pdx\leq c_1\int_{C\cap 3Q}|u|^pdx,           \label{stima2cubetto}
           \end{equation} 
           \begin{equation}        
            \int_{(2Q\times 2Q)\cap D_R} |\Phi(u)(x)-\Phi(u)(y)|^pdxdy\leq c_2\int_{(C\cap 3Q)\times (C\cap 3Q)}|u(x)-u(y)|^pdxdy,\label{stima3cubetto}
           \end{equation}
        where the positive constants $c_1$ and $c_2$ depend on $C$ and $2Q$.

     \noindent Let $(Q^\alpha_2)_{\alpha\in\Z^d}$ be the 
      open cover of $\R^d$ obtained by translating the cube $2Q$ by the vector $\alpha\in\mathbb{Z}^d$. For every set $\Omega\subset\R^d$, for every $\alpha\in\mathbb{Z}^d$ and for every real number $h>0$, we use the notation 
               \begin{equation}
               \label{omegaalphah}
               \Omega^\alpha_h := \alpha +h\Omega.
               \end{equation} 
         For $h=1$ we simply write $\Omega^\alpha =\Omega^\alpha_1$, while, for $\alpha=0$, $\Omega_h= \Omega_h^0$. 
         For every set $A\subseteq\R^d$, we define the set
            $$
            I(A):=\{\alpha\in\Z^d\hspace{0.03cm}:\hspace{0.03cm}  Q^\alpha_2\cap A\neq \emptyset\}.
            $$
      Since $\dist(\Omega', \partial\Omega)>\tilde{C}= 2\sqrt{d}k$, for every $\alpha\in I(\omp)$, we have that $Q^\alpha_{2k}\subset \Omega$.\\ 
      For any $\alpha\in I(\omp)$, we define the extension operator  $\Phi^\alpha \hspace{0.02cm}:\hspace{0.02cm} L^p(C^\alpha)\rightarrow L^p(Q^\alpha_2 )$   by translating the operator $\Phi$ by the integer vector $\alpha$. In other words, for any $u\in L^p(C^\alpha)$,
             \begin{equation}
             \label{Phialpha}
             \Phi^\alpha (u) := \left( \Phi (u\circ \pi^\alpha)      \right)\circ\pi^{-\alpha},
             \end{equation}
         where, for every $\alpha\in\mathbb{Z}^d$ and for every real number $h>0$, we use the notation
             \begin{equation}
             \label{pialphah}
             \pi^\alpha_h(x):=\alpha+hx \qquad\mbox{for\ }\hspace{0.03cm}x\in\R^d.
             \end{equation}
         If $h=1$, we write $\pi^\a=\pi^\a_1$ and if $\a=0$, we set $\pi_h=\pi^0_h$.  
        For simplicity, for $u\in L^p(\Omega\cap E)$ we denote by $u^\alpha$ the function 
             \begin{equation}
             \label{ualpha}
             u^\alpha:= \Phi^\alpha(u_{|C^\alpha})\in L^p(Q_2^\alpha). 
             \end{equation}
        From \eqref{defop} and \eqref{Phialpha}, the explicit expression of $\ua$ is given by 
             \begin{equation*}
             \ua(x):=
             \begin{dcases}
             u_{|_{C^\alpha}}(x), & x\in (2Q\cap C)^\alpha,\\
             \varphi(x-\alpha)u(\calR^{-1}(x-\alpha)+\alpha) + (1-\varphi(x-\alpha))u_{(3Q\cap C)^\alpha},& x\in (2Q\cap A_t)^\alpha,\\
             u_{(3Q\cap C)^\alpha},& x\in (2Q\setminus(C\cup A_t))^\alpha,
             \end{dcases}
             \end{equation*}
       where $A_t$ is given by (\ref{At}) with $B=C^\alpha$, $\omega={3Q}^\alpha$, and  $u_{(3Q\cap C)^\alpha}$ is the mean value of $u_{|_{C^\alpha}}$ over  $(3Q\cap C)^\alpha$;  \textit{i.e.},
            \begin{equation*}
            u_{(3Q\cap C)^\alpha} := \int_{(3Q\cap C)^\alpha} u_{|_{C^\alpha}} (x)dx.
            \end{equation*}
          
\noindent We now define the global extension operator  $L\hspace{0.02cm}:\hspace{0.02cm} L^p(\Omega\cap E)\rightarrow L^p(\omp)$. To this end, let $(\psi^\alpha)_{\alpha\in\Z^d}$ be a partition of unity  associated to $(Q^\alpha_2)_{\alpha\in\Z^d}$ such that $\psi^\beta=\psi^\alpha\circ\pi^{\alpha-\beta}$, for any $\alpha, \beta\in\mathbb{Z}^d$. Then, the map $L\hspace{0.03cm}:\hspace{0.03cm}L^p(\Omega\cap E)\to L^p(\Omega')$ is defined by 
              \begin{equation*}
              Lu:=\sum_{\alpha\in I(\omp)} u^\alpha\psi^\alpha, 
              \end{equation*}
     where $u^\alpha$ is given by \eqref{ualpha}. Note that $L$ is a linear and continuous operator from $L^p(\Omega\cap E)$ to $L^p(\Omega')$ and that condition \eqref{Iestimate} is satisfied. Indeed, in view of  \eqref{ualpha} and due to \eqref{stima1cubetto}, we  have
            \begin{equation*}
            Lu(x) =\sum_{\alpha\in  I(\Omega')}u^\alpha(x)\psi^\alpha(x) = \sum_{\alpha\in I(\Omega')}u(x)\psi^\alpha(x) = u(x)
            \end{equation*}     
     for a.e.~$x\in \Omega'\cap E$. \\Now, we show  condition \eqref{IIestimate}. To this end, fix $\beta\in I(\omp)$ and note that, for any $\alpha\in I(Q^\beta_2)$, we have $Q^\alpha_k\subset Q^\beta_{2k}$. Combined with estimate \eqref{stima2cubetto} and Jensen's inequality, this implies that, for any $u\in L^p(\Omega\cap E)$,
           \begin{eqnarray*}
           \int_{Q^\beta_2} |Lu|^pdx 
           & 
           \leq& N^{p-1}\sum_{\alpha\in I(Q^\beta_2)}\int_{Q^\beta_2\cap Q^\alpha_2} |u^\alpha|^pdx             \leq c_1N^{p-1}\sum_{\alpha\in I(Q^\beta_2)} \int_{(C\cap 3Q)^\alpha} |u|^pdx\\
            &\leq& c_1N^{p-1}\sum_{\alpha\in I(Q^\beta_2)} \int_{Q^\alpha_k\cap E} |u|^pdx
            \leq c_1N^{p} \int_{Q^\beta_{2k}\cap E} |u|^pdx,
           \end{eqnarray*}
       where $N$ denotes, henceforth, the cardinality of the set $I(Q^\beta_2)$. Taking the sum over $\beta\in I(\Omega')$ in the previous inequality, we deduce that
             \begin{eqnarray*}
             \int_{\omp} |Lu|^pdx&\leq& \sum_{\beta\in I(\omp)} \int_{Q^\beta_2} |Lu|^pdx\\
             &\leq& c_1N^p \sum_{\beta\in I(\omp)} \int_{Q^\beta_{2k}\cap E} |u|^pdx 
             \leq N^p {(2k)^{d}} c_1 \int_{\Omega\cap E} |u|^pdx.
             \end{eqnarray*}
     {The factor $(2k)^d$ is due to the fact that each point $x\in\R^d$ is contained in at most $(2k)^d$ cubes of the form $(Q^\beta_{2k})_{\beta\in\mathbb{Z}^d}$}.\\
     To conclude the proof, it remains to show condition \eqref{stimagrad}. To this end, we state the following estimate whose proof is given in Lemma \ref{lemma:crucialest} below: for all $r>0$ there exists a positive constant $c=c(r)$ such that 
          \begin{equation}
          \label{traslstima}
           \int_{((C\cap Q_3)^\alpha )^2} |u(x)-u(y)|^pdx dy\leq c(r) \int_{(Q^\alpha_k\cap E)^2\cap D_r} |u(x)-u(y)|^pdx dy.
          \end{equation}
     Fix $\beta\in\mathbb{Z}^d$. Since 
          \begin{equation*}
          Lu(x)-Lu(y) = \sum_{\alpha\in I(Q^\beta_2)}(\ua(x)-\ua(y))\psia(x)-\sum_{\alpha\in I(Q^\beta_2)}\ua(y)(\psia(y)-\psia(x))
          \end{equation*}
     for a.e.~$x,y\in Q^\beta_2$, an application of Jensen's inequality leads to
           \begin{align}
           \int_{(Q^\beta_2)^2\cap D_R }& |Lu(x) - Lu(y)|^pdx dy \notag\\ 
           &\leq 2^{p-1}\int_{(Q^\beta_2)^2\cap D_R } |\sum_{\alpha\in I(Q^\beta_2)} (\ua(x)-\ua(y))\psia(x) |^pdx dy\notag\\
           &\quad + 2^{p-1}\int_{(Q^\beta_2)^2\cap D_R } |\sum_{\alpha\in I(Q^\beta_2)}\ua(y)(\psia(y)-\psia(x))|^pdx dy.\label{stimaQbeta2}
           \end{align}
      Due to Jensen's inequality and in view of \eqref{stima3cubetto} and \eqref{traslstima}, the first integral is estimated as follows 
           \begin{eqnarray}\nonumber
          &&\hskip-2cm \int_{(Q^\beta_2)^2\cap D_R } |\sum_{\alpha\in I(Q^\beta_2)} (\ua(x)-\ua(y))\psia(x) |^pdx dy \notag\\
           &\leq& N^{p-1}\sum_{\alpha\in I(Q^\beta_2)} \int_{(Q^\beta_2\cap Q^\alpha_2)^2\cap D_R }|\ua(x)-\ua(y)|^pdx dy\notag\\
           &\leq& N^{p-1} \sum_{\alpha\in I(Q^\beta_2)}\int_{(Q^\alpha_2\times Q^\alpha_2)\cap D_R} |\ua(x)-\ua(y)|^p dx dy\notag\\
           &\leq& c_2N^{p-1} \sum_{\alpha\in I(Q^\beta_2)}\int_{((Q_3\cap C)^\alpha)^2} |u(x)-u(y)|^p dx dy\notag\\
           &\leq& c_2c(r)N^{p-1} \sum_{\alpha\in I(Q^\beta_2)}\int_{( Q^\alpha_{k}\cap E)^2\cap D_r} |u(x)-u(y)|^p dx dy\notag\\
           &\leq& c_2c(r)N^{p} \int_{( Q^\beta_{2k}\cap E)^2\cap D_r} |u(x)-u(y)|^p dx dy.\label{firsinte}
           \end{eqnarray}
       We evaluate the second integral.
       Since ${\rm supp}(\psia)\subset Q^\alpha_2$ for any $\alpha\in\mathbb{Z}^d$, we have that, for any $x,y\in Q^\beta_2$, 
            \begin{equation*}
            \sum_{\alpha\in I(Q^\beta_2)} (\psia(x)-\psia(y)) =0,
            \end{equation*}
       which implies that
             \begin{align*}
             \sum_{\alpha\in I(Q^\beta_2)} \ua(y) (\psia(y)-\psia(x)) &= \sum_{\alpha\in I(Q^\beta_2)} \ua(y) (\psia(y)-\psia(x)) - \ub(x)\sum_{\alpha\in I(Q^\beta_2)}(\psia(y)-\psia(x)) \\
             &= \sum_{\alpha\in I(Q^\beta_2)} (\ua(y)-\ub(x))(\psia(y)-\psia(x)),
             \end{align*}
       for a.e.~$x,y\in Q^\beta_2$. Thanks to the Jensen inequality, we obtain that 
            \begin{align}
            \int_{(Q^\beta_2)^2\cap D_R}& |\sum_{\alpha\in I(Q^\beta_2)} \ua(x) (\psia(y)-\psia(x))    |^pdx dy\notag\\
            &\leq N^{p-1}\sum_{\alpha\in I(Q^\beta_2)}\int_{(Q^\beta_2\cap Q^\alpha_2)^2\cap D_R} |\ua(y)-\ub(x)|^p |\psia(y)-\psia(x)|^pdx dy\notag\\
            &\leq cN^{p-1}\sum_{\alpha\in I(Q^\beta_2)}\int_{(Q^\beta_2\cap Q^\alpha_2)^2\cap D_R} |\ua(y)-\ub(x)|^p dx dy.\label{int}
            \end{align}
In order to estimate the integral on the right-hand side of \eqref{int}, we perform computations analogous to that of Lemma \ref{lemma2.6}. The difference is that $\ua$ and $\ub$ are extensions of $u$ which belong to two different translated cubes $Q^\alpha_2$ and $Q^\beta_2$. Hence, we separately evaluate the integral on the right-hand side of \eqref{int} on the following sets, which take into account the fact that $\ua$ and $\ub$ are the extension of $u\in L^p(\omp\cap E)$ on different translated cubes,
     \begin{align*}
     S_1^{\alpha,\beta} &= (Q^\alpha_2\cap Q^\beta_2\cap C)^2\cap D_R;\\
     S_2^{\alpha,\beta} &= (((2Q\cap C)^\alpha\cap Q^\beta_2)\times (Q^\alpha_2\cap (2Q\cap A_t)^\beta)  )\cap D_R;\\
     S_3^{\alpha,\beta} &= (((2Q\cap A_t)^\alpha\cap Q^\beta_2)\times (Q^\alpha_2\cap (2Q\cap C)^\beta))\cap D_R;\\
     S_4^{\alpha,\beta} &= ( ((2Q\cap A_t)^\alpha\cap Q^\beta_2)\times (Q^\alpha_2\cap (2Q\cap A_t)^\beta))\cap D_R;\\
     S_5^{\alpha,\beta} &= (((2Q\cap A_t)^\alpha\cap Q^\beta_2)\times (Q^\alpha_2\cap(2Q\setminus(C\cup A_t))^\beta))\cap D_R;\\
     S_6^{\alpha,\beta} &= ( ((2Q\setminus(C\cup A_t))^\alpha\cap Q^\beta_2)\times (Q^\alpha_2\cap (2Q\cap A_t)^\beta))\cap D_R;\\
     S_7^{\alpha,\beta} &= ( ((2Q\setminus(C\cup A_t))^\alpha\cap Q^\beta_2)\times (Q^\alpha_2\cap (2Q\setminus(C\cup A_t))^\beta))\cap D_R.
     \end{align*}
   Note that, as in Lemma \ref{lemma2.6}, the other combinations do not occur since $R$ is chosen such that $R<t$. \\
    Consider the case $(x,y)\in S_1^{\alpha,\beta}$. Since $u^\alpha = u^\beta$ a.e.~in $Q^\alpha_2\cap Q^\beta_2\cap C$ and due to estimate  \eqref{traslstima},  we have 
         \begin{align*}
          \int_{S_1^{\alpha, \beta}}|\ua(x)-\ub(y)|^p dx dy &=\int_{S_1^{\alpha, \beta}}|u(x)-u(y)|^p dx dy \\
          &\leq \int_{(2Q\cap C)^\beta\times (2Q\cap C)^\beta}|u(x)-u(y)|^p dx dy \\
          &\leq \int_{((Q_3\cap C)^\beta)^2}|u(x)-u(y)|^p dx dy \\
          &\leq c(r)\int_{(Q^\beta_{k}\cap E)^2\cap D_r}|u(x)-u(y)|^p dx dy\\
          &\leq c(r)\int_{(Q^\beta_{2k}\cap E)^2\cap D_r}|u(x)-u(y)|^p dx dy.
          \end{align*}
    Here, we have used the fact that $S^{\alpha, \beta}_1\subset (2Q\cap C)^\beta\times (2Q\cap C)^\beta$.\\
    Now, take $(x,y)\in S^{\alpha, \beta}_2$. Hence,
        \begin{align*}
        \ua(x)-\ub(y) &= u(x)-\varphi(y-\beta)u(\calR^{-1}(y-\beta)+\beta) -(1-\varphi(y-\beta))u_{(3Q\cap C)^\beta} \\
        &= [u(x) - u_{(3Q\cap C)^\alpha}] + [u_{(3Q\cap C)^\alpha} - u_{(3Q\cap C)^\beta}]\\
        &\quad \varphi(y-\beta)[u(\calR^{-1}(y-\beta)+\beta) -u_{(3Q\cap C)^\beta}],
        \end{align*}
     which implies that 
         \begin{align}
         \int_{S^{\alpha, \beta}_2} |\ua(x)-\ub(y)|^pdxdy &\leq 3^{p-1}|2Q\cap A_t|\int_{(2Q\cap C)^\alpha} |u(x) - u_{(3Q\cap C)^\alpha}|^pdx\notag\\
         &\quad + 3^{p-1}|2Q\cap C||2Q\cap A_t||u_{(3Q\cap C)^\alpha} - u_{(3Q\cap C)^\beta}|^p\notag\\
         &\quad + 3^{p-1} |2Q\cap C|\int_{(2Q\cap A_t)^\beta} |\varphi(y-\beta)|^p|u(\calR^{-1}(y-\beta)+\beta) -u_{(3Q\cap C)^\beta}|^pdy.\label{Sab2}
         \end{align}
        Taking Lemma \ref{SE} and estimate \eqref{traslstima} into account, we immediately deduce that 
             \begin{align}
             \int_{(2Q\cap C)^\alpha} |u(x) - u_{(3Q\cap C)^\alpha}|^pdx&\leq \int_{(3Q\cap C)^\alpha} |u(x) - u_{(3Q\cap c)^\alpha}|^pdx\notag\\
             &\leq \frac{1}{|3Q\cap C|}\int_{((Q_3\cap C)^\alpha)^2} |u(x) - u(y)|^pdxdy\notag\\
             &\leq \frac{c(r)}{|3Q\cap C|}\int_{(Q^\alpha_k\cap E)^2\cap D_r} |u(x) - u(y)|^pdxdy\notag\\
             &\leq \frac{c(r)}{|3Q\cap C|}\int_{(Q^\beta_{2k}\cap E)^2\cap D_r} |u(x) - u(y)|^pdxdy.\label{s1}
             \end{align}
         By \eqref{Jacobian}, we already know that  $\calR$ has bounded Jacobian and $R^{-1}(2Q\cap A_t)\subset (3Q\cap C)$. Then, in view of \eqref{traslstima}  and Lemma \ref{SE}, it follows, after the changes of variables $y'=y-\beta $ and then $y''=\calR^{-1}(y')+\beta$,  that  
             \begin{align}
            & \int_{(2Q\cap A_t)^\beta} |\varphi(y-\beta)|^p|u(\calR^{-1}(y-\beta)+\beta)-u_{(3Q\cap C)^\beta} ] |^pdy\notag\\
              &\hspace{2cm}= \int_{2Q\cap A_t} |\varphi(y')|^p|u(\calR^{-1}(y')+\beta)-u_{(3Q\cap C)^\beta} ] |^pdy'\notag\\
             &\hspace{2cm}\leq c_{\calR} \int_{(3Q\cap C)^\beta} |u(y'')-u_{(3Q\cap C)^\beta} |^pdy''\notag\\
             &\hspace{2cm}\leq \frac{c_\calR}{|3Q\cap C|}\int_{((Q_3\cap C)^\beta)^2 } |u(x)-u(y)|^pdxdy\notag\\ 
             &\hspace{2cm}\leq \frac{c_\calR }{|3Q\cap C|}c(r)\int_{(Q_k^\beta\cap E)^2\cap D_r} |u(x)-u(y)|^pdxdy\notag\\ 
             &\hspace{2cm}\leq \frac{c_1}{|3Q\cap C|}c(r)\int_{(Q_{2k}^\beta\cap E)^2\cap D_r} |u(x)-u(y)|^pdxdy.\label{s2}
             \end{align}
        In order to estimate the term $|u_{(3Q\cap C)^\alpha} - u_{(3Q\cap C)^\beta}|^p$, note that 
          \begin{align}
          |u_{(3Q\cap C)^\alpha} - u_{(3Q\cap C)^\beta}|^p &= \frac{1}{|3Q\cap C|^p}\left|\int_{(3Q\cap C)^\alpha\times (3Q\cap C)^\beta} u_{|_{C^\alpha}}(x)  -u_{|_{C^\beta}}(y)dxdy  \right|^p\notag\\
          &\leq \frac{1}{|3Q\cap C|^p} \int_{(3Q\cap C)^\alpha\times (3Q\cap C)^\beta} |u_{|_{C^\alpha}}(x)  -u_{|_{C^\beta}}(y)|^pdxdy. \label{diffmean} 
          \end{align}
       Since $u_{|_{C^\alpha}}= u_{|_{C^\beta}}$ a.e.~on $Q^\alpha_3\cap Q^\beta_3\cap C$, the last integral can be estimated as follows
         \begin{align*}
         &\int_{(3Q\cap C)^\alpha\times (3Q\cap C)^\beta} |u_{|_{C^\alpha}}(x)  -u_{|_{C^\beta}}(y)|^pdxdy\\
         &\quad = \frac{1}{|Q^\alpha_3\cap Q^\beta_3\cap C|}\int_{Q^\alpha_3\cap Q^\beta_3\cap C}\int_{(3Q\cap C)^\alpha\times (3Q\cap C)^\beta} |u_{|_{C^\alpha}}(x)-u(z)+u(z)  -u_{|_{C^\beta}}(y)|^pdxdydz\\
         &\quad\leq \frac{2^{p-1}|3Q\cap C|}{|Q^\alpha_3\cap Q^\beta_3\cap C|}\int_{(Q^\alpha_3\cap Q^\beta_3\cap C)\times (3Q\cap C)^\alpha}\ |u_{|_{C^\alpha}}(x)-u(z)|^pdxdz\\
         &\quad\quad + \frac{2^{p-1}|3Q\cap C|}{|Q^\alpha_3\cap Q^\beta_3\cap C|}\int_{(Q^\alpha_3\cap Q^\beta_3\cap C)\times (3Q\cap C)^\beta}\ |u_{|_{C^\beta}}(y)-u(z)|^pdydz.
         \end{align*}
    Since $Q^\alpha_3\cap Q^\beta_3\cap C$ is contained in $(3Q\cap C)^\alpha$, an application of estimate \eqref{traslstima} leads to 
        \begin{align*}
        \int_{(Q^\alpha_3\cap Q^\beta_3\cap C)\times (3Q\cap C)^\alpha}\ |u_{|_{C^\alpha}}(x)-u(z)|^pdxdz &\leq \int_{((Q_3\cap C)^\alpha)^2} |u(x)-u(z)|^pdxdz\\
        &\leq c(r)\int_{(Q^\alpha_k\cap E)^2\cap D_r} |u(x)-u(z)|^pdxdz\\
        &\leq c(r)\int_{(Q^\beta_{2k}\cap E)^2\cap D_r} |u(x)-u(z)|^pdxdz.
        \end{align*}
    Similarly, we also deduce that 
        \begin{align*}
        \int_{(Q^\alpha_3\cap Q^\beta_3\cap C)\times (3Q\cap C)^\beta}\ |u_{|_{C^\beta}}(y)-u(z)|^pdydz\leq c(r)\int_{(Q^\beta_{2k}\cap E)^2\cap D_r} |u(y)-u(z)|^pdydz.
        \end{align*}
     Finally,  from \eqref{diffmean} we get  
         \begin{equation}
         \label{s3}
         |u_{(3Q\cap C)^\alpha} -u_{(3Q\cap C)^\beta}|^p\leq \frac{2^pc(r)}{|3Q\cap C|^{p-1}|Q^\alpha_3\cap Q^\beta_3\cap C|}\int_{(Q^\beta_{2k}\cap E)^2\cap D_r} |u(x)-u(y)|^pdxdy.
         \end{equation}
     Gathering  estimates \eqref{s1}, \eqref{s2} and \eqref{s3}, from \eqref{Sab2} we conclude that 
        \begin{align*}
        \int_{S^{\alpha, \beta}_2}|\ua(x)-\ub(y)|^pdxdy \leq c_1(r) \int_{(Q^\beta_{2k}\cap E)^2\cap D_r} |u(x)-u(y)|^pdxdy,
        \end{align*}
      where $c_1(r)$ is a positive constant depending on $p, E$ and $r$.
      The same arguments also show that 
           \begin{align*}
           \int_{S^{\alpha, \beta}_3}|\ua(x)-\ub(y)|^pdxdy \leq c_1(r) \int_{(Q^\beta_{2k}\cap E)^2\cap D_r} |u(x)-u(y)|^pdxdy.
           \end{align*}
      Now consider $(x,y)\in S^{\alpha, \beta}_4$. We have that 
          \begin{align*}
          \ua(x)-\ub(y) &= \varphi(x-\alpha)[u(\calR^{-1}(x-\alpha)+\alpha)-u_{(3Q\cap C)^\alpha} ] + (u_{(3Q\cap C)^\alpha} - u_{(3Q\cap C)^\beta})\\
          &\quad \varphi(y-\beta)[u(\calR^{-1}(y-\beta)+\beta)-u_{(3Q\cap C)^\beta} ].
          \end{align*}
      In view of inequalities \eqref{s2} and \eqref{s3}, we obtain that
          \begin{align*}
          \int_{S^{\alpha, \beta}_4} |\ua(x)-\ub(y)|dxdx&\leq 3^{p-1}|2Q\cap A_t|\int_{(2Q\cap A_t)^\alpha} |\varphi(x-\alpha)|^p|u(\calR^{-1}(x-\alpha)+\alpha)-u_{(3Q\cap C)^\alpha}|^pdx\\ 
          &\quad 3^{p-1}|2Q\cap A_t|^2|u_{(3Q\cap C)^\alpha} - u_{(3Q\cap C)^\beta} |^p\\
          &\quad 3^{p-1}|2Q\cap A_t|\int_{(2Q\cap A_t)^\beta}|\varphi(y-\beta)|^p|u(\calR^{-1}(y-\beta)+\beta)-u_{(3Q\cap C)^\beta}|^pdy\\
          &\leq c_1(r)\int_{(Q^\beta_{2k}\cap E)^2\cap D_r} |u(x)-u(y)|^pdxdy,
          \end{align*}
      where $c_1$ is a positive constant depending on $p, E$ and $r$.\\
      Now, consider $(x,y)\in S^{\alpha, \beta}_5$. Hence,
          \begin{equation*}
          \ua(x)-\ub(y) = \varphi(x-\alpha)[u(\calR^{-1}(x-\alpha)+\alpha) - u_{(3Q\cap C)^\alpha}] + (u_{(3Q\cap C)^\alpha}- u_{(3Q\cap C)^\beta} ),
          \end{equation*}
      which, thanks to \eqref{s2} and \eqref{s3}, implies that  
            \begin{equation*}
            \int_{S^{\alpha, \beta}_5} |\ua(x)-\ub(y)|dxdx\leq c(r)\int_{(Q^\beta_{2k}\cap E)^2\cap D_r} |u(x)-u(y)|^pdxdy.
            \end{equation*}
      Similarly, if $(x,y)\in S^{\alpha, \beta}_6$, we have
                 \begin{equation*}
                  \int_{S^{\alpha, \beta}_6} |\ua(x)-\ub(y)|dxdx\leq c(r)\int_{(Q^\beta_{2k}\cap E)^2\cap D_r} |u(x)-u(y)|^pdxdy.
                  \end{equation*}
      If $(x,y)\in S^{\alpha, \beta}_7$, then \eqref{s3} shows the desired inequality on $S^{\alpha, \beta}_7$.
      Finally, gathering all the previous estimate on $S^{\alpha, \beta}_i$, $i=1,\dots, 7$ ,  from \eqref{int} it follows that 
          \begin{align*}
          \int_{(Q^\beta_2)^2\cap D_R}& |\sum_{\alpha\in I(Q^\beta_2)} \ua(y) (\psia(x)-\psia(y))    |^pdx dy\leq c_2(r)\int_{(Q^\beta_{2k}\cap E)^2\cap D_r} |u(x)-u(y)|^p dx dy,
          \end{align*}
     where $c_2$ denotes a positive constant depending on $E$, $p$ and $r$.
     In view of \eqref{stimaQbeta2}, the previous estimate combined with \eqref{firsinte} leads us to  
           \begin{equation*}
           \int_{(Q^\beta_2\times Q^\beta_2)\cap D_R} |Lu(x)-Lu(y)|^pdx dy \leq c_2(r)\int_{(Q^\beta_{2k}\cap E)^2\cap D_r} |u(x)-u(y)|^p dx dy,
           \end{equation*}
       with $c_2(r)$ being a positive constant depending on $p, E$ and $r$. Finally, summing up  over $\beta \in I(\Omega')$ in the last inequality, we conclude the
            \begin{align*}
            \int_{(\omp\times\omp)\cap D_R} |Lu(x)-Lu(y)|^pdx dy &\leq  \sum_{\beta\in I(\omp)}\int_{(Q^\beta_2\times Q^\beta_2)\cap D_R} |Lu(x)-Lu(y)|^pdx dy\\
            &\leq c_2(r) \sum_{\beta\in I(\omp)}\int_{(Q^\beta_{2k}\cap E)^2\cap D_r} |u(x)-u(y)|^pdx dy\\
            &\leq { (2k)^{2d}}c_2(r)\int_{(\Omega\cap E)^2\cap D_r}|u(x)-u(y)|^pdx dy,
            \end{align*}
       where $c_2(r)$ denotes the positive constant depending on $p$, $E$ and $r$ and the factor {$(2k)^{2d}$ is due to the fact that each point $(x,y)\in\R^d\times\R^d$ is contained in at most $(2k)^{2d}$ cubes of the form $(Q^\beta_{2k}\times Q^\beta_{2k})_{\beta\in \mathbb{Z}^d}$}. This concludes the proof.
      \end{proof}

The next result proves estimate \eqref{traslstima}.
\begin{lemma}
\label{lemma:crucialest}
Let $C$ be the connected component of $kQ\cap E$, $k\geq 4$, such that $3Q\cap E\subset C$ and $C$ has Lipschitz boundary at each point of $\partial C\cap 3\overline{Q}$. For any $r>0$ there exists a constant $c(r)>0$ such that the following inequality holds
     \begin{equation}
     \label{stimadiag}
     \int_{(3Q\cap C)^2} |u(x)-u(y)|^pdx dy\leq c(r) \int_{(kQ\cap E)^2\cap D_r} |u(x)-u(y)|^pdxdy.
     \end{equation}
\end{lemma}      
\begin{proof}
We adapt the proof of \cite[Lemma 3.3]{BP_per}.\\
Note that for any function $u$ the integral on the right-hand  side of \eqref{stimadiag} is an increasing function of $r$. Hence, it is sufficient to prove \eqref{stimadiag} for $r>0$ small enough.  For fixed $r>0$, there exists $r_1\in \left(0, {\tfrac{1}{3}r}\right)$ and $\nu\in (0,1]$ which depends on the Lipschitz constant of $\partial C\cap 3\overline{Q}$ such that for any two points $\eta', \eta''\in 3Q\cap C$ there exists a discrete path from $\eta'$ to $\eta''$;  \textit{i.e.}, a set of points 
    \begin{equation*}
    \eta_0=\eta', \eta_1, \dots, \eta_N, \eta_{N+1}=\eta''
    \end{equation*}
such that 
    \begin{itemize}
    \item[i) ] $|\eta_{j+1}-\eta_{j}|\leq r_1$, for $j=0,1,\dots, N$;
    \item[ii) ] for any $j=1, \dots, N$ the ball $B_{\nu r_1}(\eta_j)=\{\eta\in\R^d\hspace{0.03cm}:\hspace{0.03cm} |\eta-\eta_j|\leq \nu r_1 \}$ is contained in $kQ\cap C$;
    \item[iii) ] there exists $\overline{N}=\overline{N}(r_1)$ such that $N\leq \overline{N}$ for all $\eta', \eta''\in 3Q\cap C$.
    \end{itemize}
Let $\xi_j\in B_{\nu r_1}(\eta_j)$, for $j=1,\dots, N$. Hence, thanks to the Jensen inequality and the condition $ii)$ above, we deduce, for $\eta', \eta''\in 3Q\cap C$, 
    \begin{align}
    &\int_{(3Q\cap C)\cap B_{\nu r_1}(\eta')\times (3Q\cap C)\cap B_{\nu r_1}(\eta'')} |u(\xi_0)-u(\xi_{N+1})|^pd\xi_0d\xi_{N+1} \notag\\
    &=c_d(\nu r_1)^{-dN}\int_{B_{\nu r_1}(\eta_1)}\cdots\int_{B_{\nu r_1}(\eta_N)}\int_{(3Q\cap C)\cap B_{\nu r_1}(\eta')\times (3Q\cap C)\cap B_{\nu r_1}(\eta'')} |u(\xi_0)-u(\xi_1)+u(\xi_1)-\dots\notag\\ &\hspace{7cm}-u(\xi_N)+u(\xi_N)-u(\xi_{N+1})|^pd\xi_0d\xi_{N+1}d\xi_N\dots d\xi_1 \notag\\
    &\leq (N+1)^{p-1}c_d(\nu r_1)^{-dN}\int_{(kQ\cap E)\cap B_{\nu r_1}(\eta_0)}\cdots \int_{(kQ\cap E)\cap B_{\nu r_1}(\eta_{N+1})}\sum_{j=1}^{N+1}|u(\xi_j)-u(\xi_{j-1})|^pd\xi_0d\xi_{N+1}\dots d\xi_1\notag\\
    &=c(N+1)^{p-1}\sum_{j=1}^{N+1}\int_{(kQ\cap E)\cap B_{\nu r_1}(\eta_{j})\times (kQ\cap E)\cap B_{\nu r_1}(\eta_{j-1})}|u(\xi_j)-u(\xi_{j-1})|^pd\xi_jd\xi_{j-1} .\label{firstine}
    \end{align}
In view of assumption (i),  for $\xi_{j-1}\in (kQ\cap E)\cap B_{\nu r_1}(\eta_{j-1})$ and $\xi_{j}\in(kQ\cap E)\cap B_{\nu r_1}(\eta_{j})$, we have 
    \begin{align*}
    |\xi_j-\xi_{j-1}| &\leq |\xi_j-\eta_j|+|\eta_j-\eta_{j-1}| + |\eta_{j-1}-\xi_{j-1}|\leq 2\nu r_1+r_1\leq r,
    \end{align*}
which implies that $(kQ\cap E)\cap B_{\nu r_1}(\eta_{j})\times (kQ\cap E)\cap B_{\nu r_1}(\eta_{j-1})$ is contained in $ (kQ\cap E)^2\cap D_r$. In view of \eqref{firstine} and due to item (iii), we get
     \begin{align*}
     & c(N+1)^{p-1}\sum_{j=1}^{N+1}\int_{(kQ\cap E)\cap B_{\nu r_1}(\eta_{j})\times (kQ\cap E)\cap B_{\nu r_1}(\eta_{j-1})}|u(\xi_j)-u(\xi_{j-1})|^pd\xi_jd\xi_{j-1}\\
     &\qquad\leq c(N+1)^{p-1}\sum_{j=1}^{N+1}\int_{(kQ\cap E)^2\cap D_r}|u(\xi)-u(\eta)|^pd\xi d\eta\\
     &\qquad\leq c(N+1)^p \int_{(kQ\cap E)^2\cap D_r}|u(\xi)-u(\eta)|^pd\xi d\eta\\
     &\qquad\leq c(\overline{N}+1)^p \int_{(kQ\cap E)^2\cap D_r}|u(\xi)-u(\eta)|^pd\xi d\eta.
     \end{align*}
This implies that 
   \begin{align*}
   &\int_{(3Q\cap C)\cap B_{\nu r_1}(\eta')\times (3Q\cap C)\cap B_{\nu r_1}(\eta'')} |u(\xi_0)-u(\xi_{N+1})|^pd\xi_0d\xi_{N+1}\\ 
   &\hspace{2cm}\leq c(\overline{N}+1)^p \int_{(kQ\cap E)^2\cap D_r}|u(\xi)-u(\eta)|^pd\xi d\eta.
   \end{align*}
Covering $3Q\cap C$ with a finite number of balls of radius $\nu r_1$ and summing up the last inequality over all pairs of these balls gives the desired estimate \eqref{stimagrad}. 
\end{proof}

Now, we may prove Theorem \ref{ext-thm}.

\begin{proof}[Proof of Theorem {\rm\ref{ext-thm}}]
The proof follows the lines of that of Theorem 2.1 in \cite{ACDP}.\\
Fix $\e>0$ and set  $k_0=2\tilde{C}$.  First, let us show that there exist  $R=R(E)>0$, independent of $\e$, and a linear and continuous extension operator 
 $
              L_\e\hspace{0.02cm}:\hspace{0.02cm} L^p(\Omega\cap \e E)\rightarrow L^p(\Omega(\e k_0/2))
              $
      such that, for all $r>0$ and for any $u\in L^p(\Omega\cap\e E)$,
              \begin{equation}\label{a}
             L_\e (u)=u \quad\text{a.e.~in}\hspace{0.2cm}\Omega(\e k_0/2)\cap \e E,
             \end{equation}
              \begin{equation}\label{b}
              \int_{\Omega(\e k_0/2)} |L_\e (u)|^pdx\leq c_1\int_{\Omega \cap \e E} |u|^pdx,
              \end{equation}
                 \begin{equation}\label{c}
                 \int_{(\Omega(\e k_0/2))^2\cap D_{\e R}} |L_\e (u)(x)-L_\e (u)(y)|^pdx dy\leq c_2(r) \int_{(\Omega\cap \e E)^2\cap D_{\e r}}|u(x)-u(y)|^pdx dy.
                 \end{equation}
              To this end, note that for every $u\in L^p(\Omega\cap \e E)$, we have $u\circ \pi_\e\in L^p(\e^{-1}\Omega\cap E)$, where we use the notation   \eqref{pialphah} for the map $\pi_\e$. Moreover,  $\dist(\e^{-1}\Omega(\e k_0/2),\partial (\e^{-1}\Omega))>k_0=2\tilde{C}$. Hence, we can apply Lemma \ref{lemma2.7}, so that there exist $R= R(E)>0$,  independent of $\e$, and a linear and continuous operator
              $L\hspace{0.02cm}:\hspace{0.02cm} L^p(\e^{-1}\Omega\cap E)\rightarrow L^p(\e^{-1}\Omega(\e k_0/2))
              $ such that, for all $r>0$ and for all $u\in L^p(\e^{-1}\Omega\cap E)$, 
               $$L (u)=u, \quad\text{a.e.~in}\hspace{0.3cm} \e^{-1}\Omega(\e k_0/2)\cap  E,$$
              $$\int_{\e^{-1}\Omega(\e k_0/2)} |L (u)|^pdx\leq c_1\int_{\e^{-1}\Omega \cap  E} |u|^pdx, $$
             $$\int_{(\e^{-1}\Omega(\e k_0/2))^2\cap D_{R}} |L (u)(x)-L (u)(y)|^pdx dy\leq c_2(r) \int_{(\e^{-1}\Omega\cap  E)^2\cap D_{ r}}|u(x)-u(y)|^pdx dy,$$
             where the constants $c_1$ and $c_2$ are given by Lemma \eqref{lemma2.7} and they are, in particular, independent of $\e$.  Hence, we set $L_\e u=(L(u\circ \pi_\e))\circ \pi_{1/\e}$. Note that  $L_\e u\in L^p(\Omega(\e k_0/2))$ and \eqref{a}, \eqref{b}, \eqref{c} are satisfied.\\  Now, we define the extension operator $T_\e:L^p(\Omega\cap\e E)\to L^p (\Omega)$ by $T_\e(u) :=L_\e(u)$  a.e.~in $\Omega(\e k_0)$ and extended by zero out of $\Omega(\e k_0)$.  Hence, we have that $T_\e (u) \in L^p(\Omega)$ and \eqref{1e}, \eqref{2e} and \eqref{3e}  follow directly from \eqref{a}, \eqref{b} and \eqref{c} and this concludes the proof.
\end{proof}

\subsection{Compactness}\label{sect-comp}
In this section we prove a compactness result which in particular implies the equi-coerciveness of families of non-local functionals as those in the homogenization result in the next section.  The proof is based on the extension Theorem \ref{ext-thm} and on the following compactness result  proved in \cite{BP_ran} for the case $p=2$ and in \cite{AABPT} for general $p>1$.
\begin{thm}\label{comp} Let $\Omega$ be an open set with Lipschitz boundary, and assume
that for a family $\{w_\e\}_{\e >0},$ $w_\e\in L^p(\Omega),$ the estimate
\begin{equation}\label{bound}
\int_{\Omega(\e k)} \int_{D_R}\left|\frac{w_\e(x+\xi)-w_\e(x)}{\e}  \right|^p\, d\xi  \, dx\le c
\end{equation}
is satisfied with some $k > 0$ and $ R> 0.$ Assume moreover that the family $\{w_\e\}$ is bounded in  $L^p(\Omega).$
Then for any sequence $\e_j\to 0$  as $j\to +\infty$, and for any open subset $\Omega'\subset\subset \Omega$
the set  $\{w_{\e_j}\}_{j\in\N}$  is relatively compact in $L^p(\Omega')$ and every its limit point is in $W^{1,p}(\Omega)$.
\end{thm}
\begin{corollary}\label{equico} Let  $u_\e$ be { a family  of functions in $L^p(\Omega\cap\e E)$} such that {there exists $c>0$ and $r>0$} such that $||u_\e||_{L^p(\Omega\cap\e E)}\le c$ and  
\begin{equation} \label{Fvar}
   \int_{{\{|\xi|\leq r\} }}\int_{(\Omega\cap \e E)_\e (\xi)} \Bigl|\frac{u_\e(x+\e\xi)-u_\e(x)}{\e}\Bigr|^pdx\,d\xi\le c,
   \end{equation}
for all $\e>0$, with $(\Omega\cap \e E)_\e (\xi)=\{x\in\Omega\cap \e E :x+\e\xi\in\Omega\cap \e E \}$. Then, for any sequence $\e_j\to 0$  as $j\to +\infty$, and for any open subset $\Omega'\subset\subset \Omega$
the set  $\{T_{\e_j}u_{\e_j}\}_{j\in\N}$  is relatively compact in $L^p(\Omega')$ and every its limit point is in $W^{1,p}(\Omega)$.
\end{corollary}
\proof
Let  $u_\e$  be such that $||u_\e||_{L^p(\Omega\cap\e E)}\le c$ and  \eqref{Fvar} hold for every $\e>0$. 
From Theorem \ref{ext-thm}, the extended functions $T_\e u_\e$ satisfy the estimates
\begin{equation}\label{bound1}
\int_{\Omega(\e k_0)}|T_\e u_\e|^p\, dx\le c
\end{equation}
and
\begin{eqnarray*}&&
\frac{1}{\e^{d+p}}\int_{(\Omega(\e k_0))^2\cap D_{\e R}}
\left|T_\e u_\e(y)-T_\e u_\e(x)  \right|^p\, dy \, dx\\
&&\qquad\qquad \le { c(r) \int_{|\xi|\leq r}\int_{(\Omega\cap E)_\e(\xi)} \left|\frac{u_\e(x+\e\xi)-u_\e(x)}{\e}\right|^pdx\,d\xi}\leq c\,,
\end{eqnarray*}
for some $R>0$ independent of $\varepsilon$.
The latter, after the change of variables $y=x+\e\xi$, is equivalent to
\begin{equation}\label{bound2}
\int_{\Omega(\e k_0)}\int_{ |\xi|\le R}
\left|\frac{T_\e u_\e(x+\e\xi)-T_\e u_\e(x)}{\e}  \right|^p\, d\xi \, dx\le c,
\end{equation}
which corresponds to \eqref{bound}, for  $w_\e=T_\e u_\e$.
Using Theorem \ref{comp} for $w_\e=T_\e u_\e$ and \eqref{bound1}, \eqref{bound2}, we can conclude that for any sequence $\e_j\to 0$  as $j\to +\infty$, and for any open subset $\Omega'\subset\subset \Omega$, $T_{\e_j} u_{\e_j} $ is relatively   compact
in $L^p(\Omega')$ and   every its limit point is in $W^{1,p}(\Omega)$.\qed

\begin{rmk} \rm The limit $u$ in the previous corollary does not depend on the choice of the extension. In fact, 
if $\tilde v_\e$ is another extension of $u_\e$ and $v$ is its limit, then for any $\Omega''\subset\subset \Omega' \subset\subset \Omega$
$$
\int_{\Omega''\cap \e E}|u-v|^p\, dx\le c\int_{\Omega'}|u-\tilde u_\e|^p\, dx + c\int_{\Omega'}|\tilde v_\e-v|^p\, dx
$$
Passing to the limit as $\e\to 0$, one gets
$$
|(0,1)^d\cap E|\int_{\Omega''}|u-v|^p\, dx\le 0
$$
and concludes that $u=v$, by the arbitrariness of $\Omega''$.
\end{rmk}

\section{An application to homogenization}\label{sect-hom}
In this section we present an application of the Extension Theorem \ref{ext-thm} to the homogenization of non-local functional. Specifically, we consider a periodic integrand $h\hspace{0.02cm}:\hspace{0.02cm}\rd\times\rd\times{\mathbb{R}^m} \to [0, +\infty)$; {\it i.e.}, a Borel function such that $h(\cdot,\xi,z)$ is $[0,1]^d$-periodic for all $\xi\in\rd$ and $z\in\mathbb{R}^m$ and satisfies the following growth conditions: there exist positive constants $c_0, c_1, r_0$ and non-negative function $\ \psi:\mathbb{R}^d\rightarrow [0,+\infty)$ such that
      \begin{align}
      h(x,\xi,z)&\leq \psi(\xi)(|z|^p+1)\label{cre1}\\
       h(x,\xi,z)&\geq c_0(|z|^p-1)\label{cre3} \quad\forall |\xi|\le r_0
      \end{align}
with
    \begin{equation}\label{cre3.bis}
    \int_{\R^d}\psi(\xi)(|\xi|^p+1)\, d\xi \le c_1.
    \end{equation}  
Let $\Omega\subset\rd$ be an open set with Lipschitz boundary. 
For any $\varepsilon>0$, we introduce the non-local functional $H_\e\hspace{0.02cm}:\hspace{0.02cm} L^p(\Omega; \mathbb{R}^m)\to [0,+\infty]$ defined as
   \begin{equation}
   \label{Hvare}
   H_\e(u) = \int_{\rd}\int_{(\Omega\cap \e E)_\e (\xi)} h\left({x\over\e},\xi,\frac{u(x+\e\xi)-u(x)}{\e}\right)dx\,d\xi,
   \end{equation}
where for each set $B$, $\e>0$ and $\xi\in\R^d$, we use the notation 
\begin{equation}\label{Bex}
B_\e(\xi)= \{x\in B: x+\e\xi\in B\}
\end{equation}
Note that the integration in \eqref{Hvare} is performed  for $x,\xi$ such that both $x$ and $x+\e\xi$ belong to the perforated domain $\Omega\cap \e E $. 
Conditions \eqref{cre1}--\eqref{cre3.bis} guarantee that functionals $H_\e$ are estimated from above and below by functionals of the type \eqref{funperiod}.



{Thanks to Corollary \ref{equico}, our functionals $H_\e$ are equi-coercive with respect to the $L^p_\loc(\Omega)$-convergence upon identifying functions with their extensions from the perforated domain. More precisely, from each sequence $\{u_\e\}$ with equi-bounded energy $H_\e(u_\e)$ we can
extract a subsequence such that the corresponding extensions converge in $L^p_\loc$
to some limit $u\in W^{1,p}(\Omega)$. This is implied by  Corollary \ref{equico} applied with $r=r_0$ to each component of the vector-valued functions $u_\e$, upon noting that \eqref{cre3} implies \eqref{Fvar}. }\\
We now may state the homogenization result for the functional $H_\e$ with respect to the $L^p_{\rm loc}(\Omega;\R^m)$ convergence. 

\begin{thm}
\label{thm:Homthm}
The functionals $H_\e$  defined by \eqref{Hvare}  $\Gamma$-converge with respect to $L^p_{\rm loc}(\Omega;\R^m)$-convergence to the functional
     \begin{equation}\label{Hhom}
     H_\hom(u)=\left\{\begin{array}{ll}
                        \displaystyle\int_\Omega h_\hom(Du(x))\, dx & {\rm \ if\ } u\in W^{1,p}(\Omega;\mathbb R^m) \\\\
                        +\infty & {\rm \ otherwise,}
                      \end{array}
     \right.
     \end{equation}
with $ h_\hom$ satisfying the asymptotic formula 
 \begin{eqnarray}\nonumber\label{hhom}   h_\hom(\Xi)&=&\lim_{T\to+\infty}\frac{1}{T^d}\inf\Bigl\{\int_{(0,T)^d\cap E}\int_{(0,T)^d\cap  E}h(x,y-x, v(y)-v(x))\,  dx\, dy: \\
&&  v(x)=\Xi x \mbox{ \rm if }\hbox{\rm dist}(x,\partial (0,T)^d)<{k_0}    \Bigr\}
\end{eqnarray}
for all $\Xi\in \mathbf{M}^{m\times d}$.
Furthermore, if $h$ is convex in the third variable, the cell-problem formula
        \begin{eqnarray}\label{from-per}
h_\hom(\Xi)=\inf\Bigl\{\int_{(0,1)^d\cap E}\int_{E}h(x,y-x, v(y)-v(x))\,  dx\, dy:  v(x)-\Xi x \,\mbox{ is }\, 1\hbox{-}\mbox{periodic}\Bigr\}
\end{eqnarray}
holds.
\end{thm}

\proof
In \cite{AABPT} this theorem is proved when $E=\R^d$. We will prove Theorem \ref{thm:Homthm} reducing to that case by a perturbation argument. For every $\delta\ge0$ we set 
\begin{equation*}\label{accadelta}
h^\delta(x,\xi,z)=
\chi_E(x)\chi_E(x+\xi)
\,h(x,\xi,z)+\delta\chi_{B_{R_0}}(\xi)|z|^p,
\end{equation*}
where 
$R_0>0$ 
is fixed but arbitrary, and 
  \begin{equation*}
   \label{Hvaredelta}
   H^\delta_\e(u) = \int_{\rd}\int_{\Omega_\e (\xi)} h^\delta\left({x\over\e},\xi,\frac{u(x+\e\xi)-u(x)}{\e}\right)dx\,d\xi
   \end{equation*}
is defined for $u\in L^p(\Omega; \R^m )$,  where we use the notation in \eqref{Bex} for the set $\Omega_\e (\xi)$.
Note that $H^\delta_\e\ge H_\e$, and for $\delta=0$ we have $H^0_\e=H_\e$.
In the following, for any open set $A$ and $\delta\ge 0$, we also consider the `localized' functionals
$$
H^\delta_{\e}(v,A) = \int_{\rd}\int_{A_\e (\xi)} h\left({x\over\e},\xi,\frac{u(x+\e\xi)-u(x)}{\e}\right)dx\,d\xi,
$$
where we use the notation in \eqref{Bex} for the set $A_\e (\xi)$. If $\delta=0$ we write $H_{\e}(v,A)$ in the place of $H^0_{\e}(v,A)$.

The homogenization theorem in \cite{AABPT} ensures that for all $\delta>0$ there exists the $\Gamma$-limit
$$
H^\delta_\hom(u)=\Gamma\hbox{-}\lim_{\e\to0} H^\delta_\e(u)
$$
with domain $W^{1,p}(\Omega;\R^m)$, on which it is represented as
$$
H^\delta_\hom(u)=\int_\Omega h^\delta_\hom(Du)\,dx.
$$
The energy density $ h^\delta_\hom$ satisfies 
 \begin{eqnarray*}\label{hhomdelta}  \nonumber
  h^\delta_\hom(\Xi)&=&\lim_{T\to+\infty}\frac{1}{T^d}\inf\Bigl\{\int_{(0,T)^d}\int_{(0,T)^d}h^\delta(x,y-x, v(y)-v(x))\,  dx\, dy: \\  
&&\qquad v(x)=\Xi x \mbox{ \rm if }\hbox{\rm dist}(x,\partial (0,T)^d)<{r}    \Bigr\}\,,
\end{eqnarray*}
for any fixed $r>0$, and 
 \begin{eqnarray*}\
 c_1(|\Xi|^p-1)\le  h^\delta_\hom(\Xi)\le c_2(1+|\Xi|^p)
 \end{eqnarray*}
 with $c_1,c_2$ independent of $\delta$, for $\delta\in[0,1]$. Note that the independence of $c_1$ from $\delta$ is an immediate consequence of the Extension Theorem. Indeed, let $u^\delta_\e\to \Xi x$ be such that
 $$ h^\delta_\hom(\Xi)=\lim_{\e\to 0}H^\delta_\e(u^\delta_\e,(0,1)^d).$$
Applying Corollary \ref{equico} with $\Omega=(0,1)^d$, we deduce that $T_\e u^\delta_\e$ converge to $\Xi x$ locally in $(0,1)^d$ (in particular the convergence is strong {\it e.g.}~in $({1\over 4},{3\over 4})^d$). Hence, using \eqref{cre3}, the Extension Theorem, and the liminf inequality of the $\Gamma$-limit (see {\it e.g.}~\cite{BD}) we have 
\begin{eqnarray*}
\lim_{\e\to 0}H^\delta_\e(u^\delta_\e,(0,1)^d)&\ge &\liminf_{\e\to 0}H_\e(u^\delta_\e,(0,1)^d)\\
&\ge&c_0\liminf_{\e\to 0}\Bigl({1\over \e^{p+d}}\int_{((0,1)^d\cap\e E)^2\cap D_{r_0}}|u^\delta_\e(x)-u^\delta_\e(y)|^pdxdy-1\Bigr)
\\
&\ge &{c_0\over c_2(r_0)}
\liminf_{\e\to 0}\Bigl({1\over \e^{p+d}}\int_{(({1\over 4},{3\over 4})^d)^2\cap D_{R}}|T_\e u^\delta_\e(x)-T_\e u^\delta_\e(y)|^pdxdy-1\Bigr)
\\
&\ge & {c_0\over c_2(r_0)}\min\Bigl\{{1\over 2^d}c_R,1\Bigr\}\,(|\Xi|^p-1),
\end{eqnarray*}
where in the last inequality we have used that
$$\Gamma\hbox{-}\lim_{\e\to 0}
{1\over \e^{p+d}}\int_{(({1\over 4},{3\over 4})^d)^2\cap D_{R}}|v(x)-v(y)|^pdxdy
=c_R\int_{({1\over 4},{3\over 4})^d}|\nabla v|^pdx,
$$
 where $c_R=\int_{\{|\xi|\le R\}}|\xi_1|^pd\xi$ (see \cite{AABPT}). 

Since $h^\delta_\hom$ is increasing with $\delta$, we may define
$$
h_0(\Xi)=\inf_{\delta>0}h^\delta_\hom(\Xi)=\lim_{\delta\to0^+}h^\delta_\hom(\Xi),
$$
and deduce (here we use the usual notation for the upper $\Gamma$-limit) that
\begin{equation}\label{Iloilo}
\int_\Omega h_0(Du)\,dx\ge \Gamma\hbox{-}\limsup_{\e\to0} H_\e(u)\,.
\end{equation}

If $u\in W^{1,p}(\Omega;\R^m)$ and $u_\e\to u$ with $\sup_\e H_{\e}(u_\e)<+\infty$ then for all fixed $\Omega'$ compactly contained in $\Omega$, if $R_0<R$, upon identifying $u_\e$ with its extension given by the Extension Theorem, 
we obtain that,
$$
\int_{\{|\xi|\leq R_0\}}\int_{(\Omega')_\e(\xi)}
\left|\frac{u_\e(x+\e\xi)-u_\e(x)}{\e}  \right|^p\, dx \, d\xi\le c,
$$
so that 
$$ \liminf_{\e\to 0} H_{\e}(u_\e) \ge \liminf_{\e\to 0} H_{\e}(u_\e,\Omega')
\ge  \liminf_{\e\to 0} H^\delta_{\e}(u_\e,\Omega')- \delta c.
$$
From this inequality we obtain (in terms of the lower $\Gamma$-limit) 
$$
\Gamma\hbox{-}\liminf_{\e\to0}H_{\e}(u)\ge \int_\Omega h_0(Du)dx
$$
by the arbitrariness of $\delta$ and $\Omega'\subset\!
\subset\Omega$.
Hence, recalling \eqref{Iloilo}, we have proved that
$$
\Gamma\hbox{-}\lim_{\e\to0}H_{\e}(u)= \int_\Omega h_0(Du)dx,
$$
and in particular that the $\Gamma$-limit exists as $\e\to0$ (no subsequence is involved) and it can be represented as an integral functional with a homogeneous integrand. Note moreover that the lower-semicontinuity of the $\Gamma$-limit implies that $h_0$ is quasiconvex (see \cite{BD}).

We now prove that $h_0$ coincides with $h_\hom$ given by the asymptotic formula. First, note that
\begin{eqnarray}\label{disdallalto}\nonumber
h_0(\Xi) &\ge&\nonumber\limsup_{T\to+\infty}\frac{1}{T^d}\inf\Bigl\{\int_{(0,T)^d\cap E}\int_{(0,T)^d\cap  E}h(x,y-x, v(y)-v(x))\,  dx\, dy: \\
&&\qquad v(x)=\Xi x \mbox{ \rm if }\hbox{\rm dist}(x,\partial (0,T)^d)<{r}    \Bigr\}\,.
\end{eqnarray}
If we take $r=k_0$, we obtain a lower bound for $h_0$. 

To prove the opposite inequality, for any diverging sequence $\{T_j\}$ we can consider (almost-)minimizers $v_j$ of the problems in \eqref{disdallalto} with $r=k_0$ and $T=T_j$. 
By Lemma \ref{lemma2.7} (applied componentwise) with $\Omega=(0,T)^d$ and $\Omega'=({k_0\over2},T_j-{k_0\over2})^d$, recalling that $k_0=2\widetilde C$, we can consider $\widetilde v_j=L(v_j)\in L^p(({k_0\over2},T_j-{k_0\over2})^d;\R^m)$ with $\widetilde v_j=v_j$ on $\Omega=(0,T)^d\cap E$ and 
             \begin{eqnarray*}&&\hskip-2cm\int_{({k_0\over2},T_j-{k_0\over2})^d\cap D_{R}} |\widetilde v_j(\xi)-\widetilde v_j(\eta)|^pd\xi d\eta\\
             &&\leq c_2(r_0) \int_{(0,T_j)^d\cap  E)^2\cap D_{r_0}}|v_j(\xi)-v_j(\eta)|^pd\xi d\eta\le  c\, T_j^d(1+|\Xi|^p)\end{eqnarray*}
             for some $c>0$ independent of $j$.
 Upon choosing a larger $k_0>2$ we may suppose that $\lfloor {k_0\over2}\rfloor+1< k_0$ so that we may consider
 $w_j\in L^p((0,T_j-n)^d;\R^m)$, where $n=2\lfloor {k_0\over2}\rfloor+2$, defined by
 $$
 w_j(x)= L(v_j)\Bigl(x+\Bigl(\lfloor {k_0\over2}\rfloor+1\Bigr)(1,\ldots,1)\Bigr)-\Bigl(\lfloor {k_0\over2}\rfloor+1\Bigr)\Xi(1,\ldots,1).
 $$
Having set $\e_j=T_j-n$ we can consider the scaled functions
$$
u_j(x) =\e_j w_j\Bigl({x\over\e_j}\Bigr).
$$
By the boundedness of the energies above and noting that there exists $c>0$ such that $w_j(x)=\Xi x$ if $x\in E$ and dist$(x,\partial (0,T_j-n)^d)<c$, upon extracting a subsequence, we may suppose that $u_j\to u$ and $u\in \Xi x+W^{1,p}_0((0,1)^d;\R^m)$. We may then use the quasiconvexity inequality for $h_0$ to obtain 
\begin{eqnarray*}
h_0(\Xi)&\le&\int_{(0,1)^d}h_0(Du)dx\\
&\le&\liminf_j H^\delta_{\e_j} (u_j, (0,1)^d)\\
&\le&\liminf_j H_{\e_j}(u_j,(0,1)^d)+ c\delta\\
&\le&\liminf_j {1\over (T_j-n)^d}H_{1}(w_j,(0,T_j-n)^d)+ c\delta
\\
&\le&\liminf_j {1\over (T_j-n)^d}H_{1}(v_j,(0,T_j)^d)+ c\delta\\
&=&\liminf_j {1\over (T_j-n)^d}\inf\Bigl\{\int_{(0,T_j)^d\cap E}\int_{(0,T_j)^d\cap  E}h(x,y-x, v(y)-v(x))\,  dx\, dy: \\
&&\qquad v(x)=\Xi x \mbox{ \rm if }\hbox{\rm dist}(x,\partial (0,T_j)^d)<{k_0}    \Bigr\}
+ c\delta\\
&=&\liminf_j {1\over T_j^d}\inf\Bigl\{\int_{(0,T_j)^d\cap E}\int_{(0,T_j)^d\cap  E}h(x,y-x, v(y)-v(x))\,  dx\, dy: \\
&&\qquad v(x)=\Xi x \mbox{ \rm if }\hbox{\rm dist}(x,\partial (0,T_j)^d)<{k_0}    \Bigr\}
+ c\delta.
\end{eqnarray*}
By the arbitrariness of $\delta$ and of the sequence $T_j$ we obtain the desired upper bound for $h_0$, which, together with \eqref{disdallalto}, proves the asymptotic formula.

\bigskip
In the convex case, again by the homogenization results in \cite{AABPT}, we may repeat the arguments used to get \eqref{disdallalto} to obtain the lower bound for  $h_0$  
        \begin{eqnarray}\label{from-per-lob}
h_0(\Xi)\ge\inf\Bigl\{\int_{(0,1)^d\cap E}\int_{E}h(x,y-x, v(y)-v(x))\,  dx\, dy:  v(x)-\Xi x \,\mbox{ is }\, 1\hbox{-}\mbox{periodic}\Bigr\}.
\end{eqnarray}
Note that this implies that the right-hand side is bounded from above by $c_2(1+|\Xi|^p)$.
%

Now, let $v$ be an (almost) minimizing function for \eqref{from-per-lob}, and set $v_\e(x)= \e v({x\over \e})$. After applying Theorem \ref{ext-thm} to any set $\Omega$ compactly containing $(0,1)^d$ to possibly redefine $v_\e$ outside $\e E$, we can suppose that $v_\e$ converge in $L^p((0,1)^d;\R^m)$ to $\Xi x$ and that 
$$
{1\over\e^{p+d}}\int_{((0,1)^d\times(0,1)^d)\cap D_{\e R_0}}\left| v_\e(x)-v_\e (y)     \right|^p\, dx\, dy\le c(1+|\Xi|^p).
$$
We then estimate
 \begin{eqnarray*}
h^{\delta}_\hom(\Xi)&\le&\liminf_{\e\to 0}H^\delta_\e(v_\e)\\
&\le& \int_{(0,1)^d\cap E}\int_{E}h(x,y-x, v(y)-v(x))\,  dx\, dy +c\delta (1+|\Xi|^p). 
\end{eqnarray*}
Taking the limit as $\delta\to0$, we obtain the converse inequality of \eqref{from-per-lob}, and conclude the proof.\qed

  \begin{rmk}\rm
  The function $h_{\rm hom}$ obtained in the asymptotic formula \eqref{hhom}  also satisfies 
      \begin{equation*}  
       h_\hom(\Xi)=\lim_{T\to+\infty}\frac{1}{T^d}\inf\left\{\int_{(0,T)^d\cap E}\int_{(0,T)^d\cap  E}h(x,y-x, v(y)-v(x))\,  dx\, dy: \right. 
      \end{equation*}
      $$ \left.  v(x)-\Xi x \,\mbox{is}\, (0,T)^d-\mbox{periodic}    \right\}.
      $$
  \end{rmk}
  \begin{rmk}\rm
  An example is given by the convolution functional  
      \begin{equation*}
      F_{\varepsilon}(u) = \frac{1}{\varepsilon^{d+p}}\int_{(\Omega\cap E_\varepsilon)\times (\Omega\cap E_\varepsilon)} a\left(\frac{y-x}{\varepsilon}\right)|u(x)-u(y)|^p\,dy\,dx.
      \end{equation*}
Since the integrand function $h(x,\xi, z) = a(\xi)|z|^p$ is convex in $z$, then Theorem \ref{thm:Homthm} and \eqref{from-per}  ensure that the integrand of the $\Gamma$-limit \eqref{Hhom} of $F_\e$ 
is given by
  \begin{equation*}\label{fhom}
  \inf\left\{\int_{(0,1)^d\cap E}\int_{E-\{x\}}a(\xi)|v(x+\xi)-v(x)|^p\, d\xi\, dx: v(x)-\Xi x\hbox{ is } 1{\rm-periodic}
  \right\}.
  \end{equation*}
  \end{rmk}

\subsection*{Acknowledgments}
Andrea Braides acknowledges the MIUR Excellence Department Project awarded to the Department of Mathematics, University of Rome Tor Vergata, CUP E83C18000100006.
Valeria Chiad\`o Piat and Lorenza D'Elia acknowledge the MIUR Excellence Department Project 2018-2022 awarded to the Department of Mathematical Sciences (DISMA) in the Politecnico di Torino.  All the authors are members of INdAM-GNAMPA.

\end{document}